\documentclass[12pt]{amsart}
\usepackage{srcltx}
\usepackage{amsmath}
\usepackage{amsfonts}
\usepackage{amssymb}
\usepackage{amsthm}
\usepackage[T1]{fontenc}
\usepackage{hhline}                 
\usepackage{multirow}                 
\usepackage{array}                    
\usepackage{float}  
\usepackage{mathtools}
\usepackage{latexsym}
\usepackage{enumitem}
\usepackage[hidelinks]{hyperref}
\usepackage{amsrefs}
\usepackage{array,multirow,makecell}
\newcolumntype{C}[1]{>{\centering\arraybackslash}p{#1}}

\def\Z{\mathbb{Z}}

\def\F{\mathbb{F}}

\newcommand{\supp }{\mathrm{supp\,}}

\newcommand{\Jac }{\mathrm{ {\mathcal J}ac}}

\newtheorem{theo}{Theorem}[section]
\newtheorem{definition}[theo]{Definition}
\newtheorem{proposition}[theo]{Proposition}
\newtheorem{corollary}[theo]{Corollary}
\newtheorem{lemma}[theo]{Lemma}
\newtheorem{remark}[theo]{Remark}
\newtheorem{theorem}[theo]{Theorem}
\newtheorem{example}[theo]{Example}

\def\div{\mathsf{D}}
\def\pp{\mathsf{Princ}}
\def\FF{\mathbf{F}}
\def\CC{\mathbf{C}}
\def\bP{\mathsf{P}}

\begin{document}
	\title
	[On the non-special divisors in algebraic function fields]{On the non-special divisors
		in algebraic function fields defined over $\F_q$}
	
	\author{S. Ballet}
		
	\address{Aix Marseille Univ, CNRS, Centrale Marseille, Institut
		de Mathématiques de Marseille, Marseille, France.}
	\email{stephane.ballet@univ-amu.fr}
	
	\author{M. Koutchoukali}
		
		\address{Aix Marseille Univ, CNRS, Centrale Marseille, Institut
			de Mathématiques de Marseille, Marseille, France.}
	\email{koutchoukali.mahdi@hotmail.fr}
	

	\date{\today}
	\begin{abstract}
	In the theory of algebraic function fields and their applications to the information theory, the Riemann-Roch theorem plays a fundamental role.  
	But its use, delicate in general, is efficient and practical for applications especially in the case of non-special divisors.
	So, in this paper, we give a survey of the known results concerning the non-special divisors in the algebraic function fields defined over finite fields, enriched with some new results about the existence of such divisors in curves of defect k. In particular, we have chosen to be self-contained by giving the full proofs of each result, the original proofs or shorter alternative proofs.
	\end{abstract}
	
	\maketitle
	
	\section{Introduction}
	
	\subsection{General context}
	\sloppy
	This article mainly is a survey highlighting the current state-of-the-art on the existence and the construction of non-special divisors in algebraic function fields defined over finite fields. The growing importance of this topic has attracted many mathematicians and computer scientists, who developed new ideas and obtained new results. Finite fields constitute an important area of mathematics. They arise in many applications, particularly in areas related to information theory as for example cryptography and error-correcting codes. Moreover, the algebraic geometric codes "à la Goppa" whose performance as error-correcting codes has been proven since the results of Tsfasman-Vladuts-Zink in \cite{tsvlzi} also plays a important role in cryptography and in algorithmic (computational algebraic geometry) as demonstrated by the work carried out in recent years on sharing secret schemas (see for example \cite{chcr}, \cite{cacrxing}, \cite{cacramerxing} and \cite{tuuz}) or on the algebraic complexity of the multiplication in the finite fields (see survey \cite{bachpirararo}). In all these cases, the use of Riemann-Roch Theorem with non-special divisors of degree $d$ in algebraic function fields of genus $g$ over finite fields, in particular in the non-trivial case where $g-1\leq d\leq 2g-2$, is crucial.   
	
        \subsection{Organisation} 
         
         In this paper, we give a survey of the known results concerning the non-special divisors in the algebraic function fields defined over finite fields, 
	enriched with some unpublished recent results. In particular, we have chosen to be self-contained by giving the full proofs of each result that relates directly to the non-special divisors, the original proofs or shorter alternative proofs.
		
In Section \ref{prelim}, we present Notation and the well-known elementary results about non-special divisors in algebraic function fields defined over an arbitrary field. In Section \ref{Existence-g-Et-gmoins1}, we focus on the results concerning the existence of non-special divisors of degree $g$ and $g-1$. In particular, in section \ref{NR}, new results are proposed on the existence of these divisors in curves of defect $k$. 
Finally, in Section 4, we give known results concerning the explicit construction of non-special divisor of degree $g$ and $g-1$.

	
	\section{Preliminaries}\label{prelim}
	
	\subsection{Notation.}
	Let $\FF/\F_q$ be an algebraic function field of 
	one variable defined over a finite field $\F_q$. We will always 
	suppose that the 
	full constant field of $\FF/\F_q$ is $\F_q$ and denote by $g$ the genus of $\FF$. 
	If $D$ is a (rational) divisor we recall that  the $\F_q$-Riemann-Roch vector space associated to $D$ and denoted
	${\mathcal L}(D)$ is the following subspace of rational functions
	\begin{equation}\label{fct}
		{\mathcal L}(D)= \{ x \in \FF ~|~(x)\geq -D\}\cup \{0\}.
	\end{equation}
	By Riemann-Roch Theorem we know that the dimension of this vector space,
	denoted by $\dim (D)$, is related to the genus of $\FF$ and the degree of $D$
	by 
	\begin{equation} \label{RR}
		\dim (D) = \deg (D)-g+1 +\dim (\kappa-D),
	\end{equation}
	where $\kappa$ denotes a canonical divisor of $\FF/\F_q$. 
	In this relation, the complementary term $i(D) =\dim (\kappa-D)$ is called the index of speciality and 
	is not easy to compute in general.
	In particular, a divisor $D$ is non-special when the index of speciality $i(D)$
	is zero.
	 Let us recall the usual notation (for the basic notions related
	to an algebraic function field see \cite{stic}). 
	For any integer $k\geq 1$ we denote by $\bP_k(\FF/\F_q)$ 
	the set of places of degree $k$, by $B_k(\FF/\F_q)$ the 
	cardinality of this set and by $\bP(\FF/\F_q)=\cup_k \bP_k(\FF/\F_q)$.
	The divisor group of $\FF/\F_q$  is denoted by ${\div}(\FF/\F_q)$.
	If a divisor $D \in {\div}(\FF/\F_q)$ is such that
	$$D= \sum_{P \in \bP(\FF/\F_q)} n_P P,$$
	the support of $D$ is the following finite set
	$$\supp (D)=\{ P \in \bP(\FF/\F_q)~|~n_P \neq 0\}$$
	and its degree is
	$$\deg(D)=\sum_{P\in \bP(\FF/\F_q)} n_P\deg(P).$$
	We denote by ${\div}_n(\FF/\F_q)$ the set of divisors of degree $n$.
	We say that the divisor $D$ is effective if for each $P\in \supp(D)$
	we have $n_P \geq 0$ and we denote by ${\div}_n^+(\FF/\F_q)$ the set of effective divisors of degree $n$ and 
	by $A_{n}=\# {\div}_n^+(\FF/\F_q)$ the cardinal of the set ${\div}_n^+(\FF/\F_q)$. In general, we will note $\# \mathcal{U}$ 
	the cardinal of the set $\mathcal{U}$.
	The dimension of a divisor $D$, denoted by $\dim(D)$, is the dimension
	of the vector space ${\mathcal L}(D)$ defined by the formula (\ref{fct}).
	Let $x \in \FF/\F_q$, we denote by $(x)$ the divisor
	associated to the rational function $x$, namely
	$$(x)=\sum_{P\in\bP(\FF/\F_q)} v_P(x)P,$$
	where $v_P$ is the valuation at the place $P$.
	Such a divisor $(x)$ is called a principal divisor, and
	the set of  principal divisors is
	a subgroup of ${\div}_0(\FF/\F_q)$ denoted by ${\pp}(\FF/\F_q)$.
	The factor group 
	$${\mathcal C}(\FF/\F_q)={\div}(\FF/\F_q)/{\pp}(\FF/\F_q)$$
	is called the divisor class group. If $D_1$ and $D_2$ are
	in the same class, namely if the divisor $D_1-D_2$ is principal,
	we will write $D_1 \sim D_2$. We will denote
	by $[D]$ the class of the divisor $D$.\\
	If $D_1 \sim D_2$, the following holds
	$$\deg(D_1)=\deg(D_2), \quad \dim(D_1)=\dim(D_2),$$
	so that we can define the degree 
	$\deg([D])$ and the dimension $\dim([D])$ of a class.
	Since the degree of a principal divisor is $0$, we can define the subgroup
	${\mathcal C}(\FF/\F_q)^0$ of classes of degree $0$ divisors in ${\mathcal C}(\FF/\F_q)$.
	It is a finite group and we denote by $h$ its order, called the \emph{class number} of $\FF/\F_q$. Moreover if 
	$$L(t)=\sum_{i=0}^{2g} a_i t^i=\prod_{i=1}^{g} [(1-\alpha_i t)(1-\overline{\alpha_i} t)]$$
	with $|\alpha_i|=\sqrt{q}$ is the numerator of the Zeta function of $\FF/\F_q$, we have  
	$h=L(1)$. Finally we will denote $h_{n,k}$ the number of classes of divisors of degree $n$ and of dimension $k$.\\
	
	In the sequel, we may  simultaneously use the dual language of (smooth, absolutely irreducible, projective) curves by associating to $\FF/\F_q$ a unique ($\F_q$-isomorphism class of) curve $\CC/\F_q$ of genus $g$ and conversely to such a curve its function field $\CC(\F_q)$ is the set of $\F_q$-rational points on $\CC$, and $\F_q(\CC)$ is the field of rational functions on $\CC$ over $\F_q$. Because, by F.K. Schmidt's theorem (cf. \cite[Corollary V.1.11]{stic}) there always exists a rational divisor of degree one, the group ${\mathcal C}(\FF/\F_q)^0$ is isomorphic to the group of $\F_q$-rational points on the Jacobian of $\CC$, denoted by $\Jac(\CC)$. In particular $h(\FF/\F_q)=\# \Jac(\CC)(\F_q)$ is the cardinal of the set $\Jac(\CC)(\F_q)$.

	\subsection{Elementary results on non-special divisors.}
	Let $\FF/\F_q$ be a function field of genus $g>0$. First recall some results about non-special divisors (cf. \cite{stic}). If $\deg(D) < 0$, then $\dim(D) = 0$ and if $\deg(D) \geq 0$ then $\dim(D) \geq \deg(D)-g+1$. When $0\leq \deg(D)\leq 2g-2$, the computation of $\dim(D)$ is difficult, but we have the following general results.
	
	\begin{proposition} \label{BR}
		\begin{enumerate}
			
			\item $\F_q \subset \mathcal{L}(D)$ if and only if $D \geq 0$.
			\item If $deg(D) > 2g-2$, then $D$ is non-special.
			\item The property of a divisor $D$ being special or non-special depends only on the class of $D$ up to equivalence.
			\item Any canonical divisor $\kappa$ is special, $deg(\kappa) = 2g-2$ and $dim(\kappa) = g$.
			\item Any divisor $D$ with $dim(D) > 0$ and $deg(D) <g$ is special.
			\item If $D$ is non-special and $D' \geq D$ then $D'$ is non-special.
			\item For any divisor $D$ with $0 \leq deg(D) \leq 2g-2$, $dim(D) \leq 1+ \frac{1}{2} deg(D)$ holds. 
			
		\end{enumerate}
	\end{proposition}
	
	For the rational function field $\FF = \F_q(x)\ (g=0)$, there is no non-zero regular differential, thus, all divisors of degree $d \geq 0$ are non-special. From now, we focus on the existence of non-special divisors of degree $g$ or $g-1$. Note that $g-1$ is the least possible degree for a divisor $D$ to be non-special. We have the following trivial observations.
	
	\begin{lemma}
		Assume $g \geq 1$. Let $D \in {\div}(\FF/\F_q)$ and set $d=deg(D)$.
		\begin{enumerate}
			
			\item If $d=g$, $D$ is a non-special divisor if and only if $dim(D)=1$. Assume that $D$ is a non-special divisor of degree $g$; then, $D \sim D_0$, where $D_0$ is effective. If $D \geq 0$ and $d=g$, $D$ is non-special divisor if and only if $\mathcal{L}(D)=\F_q$.
			\item If $d=g-1$, $D$ is a non-special divisor if and only if $dim(D)=0$. A non-special divisor of degree $g-1$, if any, is never effective.
			\item If $g >1$ and $A_{g-1}=0$, then any divisor of degree $g-1$ is non-special.
			
		\end{enumerate}
		
	\end{lemma}
	
	A consequence of assertion 1 is
	
	\begin{lemma}\label{gtog-1}
		
		Assume that $D$ is an effective non-special divisor of degree $g \geq 1$. If there exists a degree one place $P$ such that $ P \notin supp(D)$, then $D-P$ is a non-special divisor of degree $g-1$
		
	\end{lemma}

	\section{Existence of non-special divisors of degree $g$ and $g-1$}\label{Existence-g-Et-gmoins1}
	Unless otherwise specified, the results in this section come from \cite{balb} by S. Ballet and D. Le Brigand. They determine the necessary and sufficient conditions for the existence of non-special divisors of degree $g$ and $g-1$ in the general case and they apply them to the cases of $g = 1$ and $g=2$.
	
	\subsection{General case}
	
	Here, we give some general results about non-special divisors of degree $g$ and $g-1$.
	
	\begin{proposition} \label{prop4}
		
		Let $\FF/\F_q$ be an algebraic function field of genus $g \geq 1$.
		
		\begin{enumerate}
			
			\item If  $B_1(\FF/\F_q) \geq g$, there exists a non-special divisor $D$ such that $D \geq 0$, $deg(D) = g$ and $supp(D) \subset  \bP_1(\FF/\F_q)$.
			\item If $B_1(\FF/\F_q) \geq g+1$, there exists a non-special divisor such that $deg(D) = g-1$ and $supp(D) \subset  \bP_1(\FF/\F_q)$.
			
		\end{enumerate}
		
	\end{proposition}

	\begin{proof}
		\begin{enumerate}
			\item cf. \cite[Proposition I.6.10]{stic}.
			\item Let $T \subset \bP_1(\FF/\F_q)$ be such that $\#T = g$ and, using assertion 1, let $D \geq 0$  be a non-special divisor such that $deg(D) = g$ and $supp(D) \subset T$. Select $P \in \bP_1(\FF/\F_q ) \setminus supp(D)$ and apply Lemma \ref{gtog-1}.
		\end{enumerate}
	\end{proof}
	
	 We denote by $\mathcal{E}_g$ and $\mathcal{E}_{g-1}$ the following holds:\\
	
	$\mathcal{E}_g$ : $\FF/\F_q$ has an effective non-special divisor of degree $g$, and \\
	$\mathcal{E}_{g-1}$ : $\FF/\F_q$ has a non-special divisor of degree $g-1$. \\
	
	If $\FF/\F_q$ has enough rational places compared to the genus, then $\mathcal{E}_g$ and $\mathcal{E}_{g-1}$ are true.
	
		\begin{proposition} \label{prop5}
		Let $\FF/\F_q$ be a function field of genus $g$. Denote by $h$ the order of the divisor class group of $\FF/\F_q$.
		\begin{enumerate}
			\item If $A_g < h(q+1)$, then $\mathcal{E}_g$ is true.
			\item If $A_{g-1} < h$, then $\mathcal{E}_{g-1}$ is true.
			\item Assume $g \geq 2$. If $A_{g-2} < h$, then $\mathcal{E}_g$ is true.
			\item If $g=2$ or $3$, $\mathcal{E}_g$ is untrue if and only if $A_{g-2}=h$.
		\end{enumerate}
	\end{proposition}
	
	\begin{proof}
		Recall that, in any function field, there exists a degree 1 divisor (this result of F.K. Schmidt; see \cite[Corollary V.1.11]{stic} for instance), so there exist divisors of any degree $ \geq 1$. Let $d \geq 1$ and $D_0 \in {\div}_d^+(\FF/\F_q)$, and consider the map $\psi_{d,D_0}$:
		\begin{equation*}
			\begin{array}{cccc} 
				\psi_{d,D_0}: & {\div}_d^+(\FF/\F_q)  & \longrightarrow & \mathcal{J}ac(\FF/\F_q) \cr
				& D &  \longmapsto & [D - D_0]
			\end{array}
		\end{equation*}
		
		\begin{enumerate}
			\item First, it is well known that $1 \leq h \leq A_g$ is true for any function field. Indeed, let $D$ be such that $deg(D)=g$. By Riemann-Roch, $dim(D) \geq 1$; thus, there exists an effective divisor of degree $g$ which is equivalent to $D$. So assume $D_0 \in {\div}_g^+(\FF/\F_q)$ and consider the map $\psi_{g,D_0}$. For all $[R] \in \mathcal{J}ac(\FF/\F_q)$, we have $deg(R+D_0)=g$; thus, $dim(R+D_0) \leq 1$ and there exists $u \in \FF^*$ such that $D:=R+D_0+div(u)$ is in $ {\div}_g^+(\FF/\F_q)$ and $[R]=[D-D_0] = \psi_{g,D_0}(D)$. This proves that $\psi_{g,D_0}$ is surjective and that $h \leq A_g$. Assume now that $\FF/\F_q$ has no non-special divisor $D$ of degree $g$. Then, $dim(D) \geq 2$ for all degree $g$ divisors; thus, for all $[R] \in \mathcal{J}ac(\FF/\F_q)$, we have
			$$ \# \{ D \in {\div}_g^+(\FF/\F_q), [D-D_0] = [R] \}  = \frac{q^{dim(R+D_0)}-1}{q-1} \geq \frac{q^2-1}{q-1}=q+1$$
			and $A_g \geq h(q+1)$.
			\item A divisor $D$ of degree $g-1$ is non-special if and only if $dim(D)=0$. If $g=1$, there exists a non-special divisor of degree $g-1=0$  if and only if $h=B_1 > 1 = A_0$, since two distinct degree one places are not equivalent. Assume now that $g >1$. Hence, it is sufficient to prove the existence of a divisor of degree $g-1$ which is not equivalent to any effective divisor. If $A_{g-1}=0$, the result is proved. otherwise, let $D_0$ be an effective divisor of degree $g-1 \geq 1$ and consider the map $\psi_{g-1,D_0}$. If $A_{g-1} < h$, this map is not surjective. Hence, there exists a zero-degree divisor $R$ such that $[R]$ is not in the image of $\psi_{g-1,D_0}$. Consequently, $D=R+D_0$ is a divisor of degree $g-1$ which is not equivalent to an effective divisor. Thus, $D$ is non-special.
			\item From the functional equation of the zeta function, it can be deduced (see \cite[Lemma 3(i)]{nixi}) that, for $g \geq 1$, one has
			$$ A_n = q^{n+1-g}A_{2g-2-n} + h \frac{q^{n+1-g}-1}{q-1} \ for\ all\ 0 \leq n \leq 2g-2. $$
			For $g \geq 2$ and $n=g$ this gives
			$$ A_g=h+qA_{g-2}. $$
			Thus, if $g \geq 2$, 
			$$ A_g < (q+1)h \Longleftrightarrow A_{g-2} <h. $$ 
			\item Assume that $g=2$ or $3$. Then if $deg(D)=g$ and $dim(D) \geq 2$, one has $dim(D)=2$ by assertion 7 of Proposition \ref{BR}, since $dim(D) \leq \frac{g}{2} +1$. Thus, $\mathcal{E}_g$ is untrue if and only if $A_g = (q+1)h$, which is equivalent to $A_ {g-2}=h$.
		\end{enumerate}
	\end{proof}
	
	We quote the following consequence of assertion 2.
	
	\begin{corollary} \label{cor6}
		Let $\FF/\F_q$ be an algebraic function field of genus $g \geq 2$ such that $A_{g-1} \geq 1$. Denote by $h$ the order of the divisor class group of $\FF/\F_q$. Then $\mathcal{E}_{g-1}$ is untrue if and only if there exists $h$ elements of  ${\div}_{g-1}^+(\FF/\F_q)$ pairwise non-equivalent.
	\end{corollary}

	\begin{proof}
		Let $r$ be the maximum number of pairwise non-equivalent elements of  ${\div}_{g-1}^+(\FF/\F_q)$ and let $D_1, \ldots, D_r$ be elements of  ${\div}_{g-1}^+(\FF/\F_q)$ pairwise non-equivalent. Then:
		$$ \{ [0] = [D_1-D_1],[D_2-D_1], \ldots, [D_r-D_1] \} $$
		is a subset of $\mathcal{J}ac(\FF/\F_q)$ of order $r$. If $r=h$, for any divisor $D$ of degree $d=g-1$, we have $[D-D_1]=[D_i-D_1]$ for some $i$, $1 \leq i \leq h$, and then $D \sim D_i$. Thus, $dim(D) \geq 1$. If $r <h$, $\psi_{g-1,D_1}$ is not surjective and the result follows.
	\end{proof}

%
%

	\subsection{Existence of non-special divisors of degree $g$}
	
	Now, we particularly focus on the non-special divisors of degree $g$. First, we give useful properties and interesting information to study the existence non-special divisors. For example, it is known that this existence is relied on the number of effective divisors of certain degrees $A_n$. 
	
	\begin{lemma}
		If $B_1 \geq m \geq 1$, then for all $n \geq 2$ one has
		\begin{equation}
			A_n \geq m A_{n-1} - \frac{m(m-1)}{2} . A_{n-2}. \label{c}
		\end{equation}
	\end{lemma}
	
	\begin{proof}
		See \cite[Lemma 4]{nixi}.
	\end{proof}
	
	Moreover, it also is linked to the class number $h$ of algebraic function fields.
	
	\begin{proposition} \label{tab}
		Let $\FF/\F_q$ be a function field of genus $g \geq 2$. We denote by $h$ its divisor class number. 
		\begin{itemize}
			\item Up to isomorphism, there are 4 function fields $\FF/\F_q$, 2 of them being hyperelliptic, such that $h=1$ and $g \geq 2$. They are obtained for $\FF = \mathbb{F}_2(x,y)$ with:\\
			
			\begin{tabular}{|*{5}{c|}}
				\hline
				$g$ & Equation & $B_1$  & $B_2$ & $B_3$\tabularnewline
				\hline
				2 & $y^2+y+(x^5+x^3+1)=0 $ & 1 & 2 &  \tabularnewline
				2 & $y^2+y+(x^3+x^2+1)/(x^3+x+1)=0$ & 0 & 3 & \tabularnewline
				\hline
				3 & $y^4+xy^3+(x+1)y+(x^4+x^3+1)=0$ & 0 & 0 & 1 \tabularnewline
				3 & $y^4+xy^3+(x+1)y+(x^4+x+1)=0$ & 0 & 1 & 1 \tabularnewline
				\hline
			\end{tabular}\\
			\item Up to isomorphism, there are 15 functions fields $\FF/\F_q$, 7 of them being hyperelliptic, such that $h=2$ and $g \geq 2$.They are obtained for $\FF = \mathbb{F}_2(x,y)$ with:
			\newpage
			\begin{table} 
				\begin{tabular}{|c|c|p{9cm}|*{4}{c|}}
					\hline
					$q$ & $g$ & Equation & $B_1$  & $B_2$ & $B_3$ & $B_4$ \tabularnewline
					\hline
					3 & 2 & $y^2-2(x^2+1)(x^4+2x^3+x+1)=0 $ & 1 & 5 & &  \tabularnewline
					\hline
					2 & 2 & $y^2+y+(x^3+x+1)/(x^2+x+1)=0$ & 1 & 3 & & \tabularnewline
					&    & $y^2+y+(x^4+x+1)/x=0$ & 2 & 1 & & \tabularnewline
					\hline
					2 & 3 & $y^2+y+(x^4+x^3+x^2+x+1)/(x^3+x+1)=0$ & 1 & 2 & 1 & \tabularnewline
					&    & $y^2+y+(x^5+x^2+1)/(x^2+x+1)=0$ & 1 & 3 & 0 & \tabularnewline
					&    & $y^2+y+(x^6+x+1)/(x^2+x+1)^3=0$ & 0 & 4 & 2 & \tabularnewline
					&    & $y^2+y+(x^4+x^3+1)/(x^4+x+1)=0$ & 0 & 3 & 2 & \tabularnewline
					&    & $y^4+xy^3+(x+1)y+(x^4+x^2+1)=0$ & 0 & 2 & 2 & \tabularnewline
					&    & $y^3+(x^2+x+1)y+(x^4+x^3+1)=0$ & 1 & 0 & 3 & \tabularnewline
					&    & $y^3+y+(x^4+x^3+1)=0$ & 1 & 1 & 2 & \tabularnewline
					\hline
					2 & 4 & $y^3+(x^4+x^3+1)y+(x^6+x^3+1)=0$ & 0 & 0 & 4 & 2 \tabularnewline
					&   & $y^3+(x^4+x^2+1)y+(x^6+x^5+1)=0$ & 0 & 0 & 4 & 2 \tabularnewline
					&   & $y^3+(x^4+x^3+1)y+(x^6+x+1)=0$ & 0 & 1 & 3 & 3 \tabularnewline
					&   & $y^6+xy^5+(x^2+1)y^4+(x^3+x^2)y^3+(x^6+x^5+x^3+x+1)=0$ & 0 & 1 & 1 & 3 \tabularnewline
					&   & $y^6+xy^5+x^3y^3+y^2+(x^5+x^2)y+(x^6+x^2+1)=0$ & 0 & 1 & 2 & 3 \tabularnewline
					\hline
				\end{tabular}
			\end{table}
		\end{itemize}
	\end{proposition}
	
	\begin{proof}
		See \cite{lemaqu} and \cite{maqu} for the solutions of the $(h=1)$ problem and \cite[Proposition 3.1 and Theorem 4.1]{lebri} for the solutions of the $(h=2)$ problem.
	\end{proof}

	\begin{proposition} \label{DeggDiv}
		An algebraic function field $\FF/\F_q$ of genus $g \geq 2$ has an effective non-special divisor of degree $g$ in the following cases: 	
		\begin{enumerate}
			\item[\rm(i)] $q \geq 3$.
			\item[\rm(ii)] $q=2$ and $g=2$, unless $\FF = \mathbb{F}_2(x,y)$ with:\\
			
			$$ y^2+y+(x^5+x^3+1)=0$$ and $$ y^2+y+(x^3+x^2+1)/(x^3+x+1)=0$$
			
			\item[\rm(iii)] $q=2$ and $g=3$.
			\item[\rm(iv)] $q=2$, $g \geq 4$ and $B_1(\FF/\F_q) \geq 3$.
		\end{enumerate}
	\end{proposition}
	
	\begin{proof}
		We set $L(t):=L(\FF/\F_q,t)$. for $g \geq 2$, it follows that (see \cite[Lemma 3]{nixi}): 
		$$ \sum_{n=0}^{g-2} A_n t^n + \sum_{n=0}^{g-1} q^{g-1-n} A_n t^{2g-2-n} = \frac{L(t) -ht^g}{(1-t)(1-qt)}.$$
		Substituting $t=q^{-1/2}$ into the last identity, we obtain:
		$$ 2 \sum_{n=0}^{g-2} q^{-n/2} A_n + q^{-(g-1)/2} A_{g-1} = \frac{h-q^{g/2} L(q^{-1/2})}{(q^{1/2}-1)^2 q^{(g-1)/2}}. $$
		and since $L(q^{-1/2})= \prod_{i=1}^{g} |1-\pi_i q^{-1/2}|^2 \geq 0$, we have: 
		\begin{equation}
			2 \sum_{n=0}^{g-2} q^{(g-1-n)/2} A_n + A_{g-1} \leq \frac{h}{(q^{1/2} - 1)^2}. \label{a}
		\end{equation}
		\begin{enumerate}
			\item $q \geq 3$. Using \eqref{a}, $A_{g-2} \geq h$ implies that:
			$$ 2 q^{1/2} \leq \frac{1}{(q^{1/2}-1)^2},$$
			which is absurd if $q \geq 3$. Thus, $A_{g-2} < h$ is always satisfied and so $\mathcal{E}_{g}$ is true.
			\item If $q=2$ and $g \geq 3$, \eqref{a} implies:
			\begin{equation}
				4 A_{g-3} + 2 \sqrt{2} A_{g-2} + A_{g-1} \leq \frac{h}{(\sqrt{2}-1)^2}= (3+2\sqrt{2})h. \label{b}
			\end{equation}
			Assume that $B_1(\FF/\F_q) \geq m = 3$; then, by \eqref{c} with $n=g-1$, we have $A_{g-1} + 3 A_{g-3} \geq 3 A_{g-2}$, and finally using \eqref{b}:
			$$ A_{g-3} + (3+2 \sqrt{2}) A_{g-2} \leq (3+2\sqrt{2})h.$$
			Since $A_{g-3} \geq 1$, because if $g=3$, $A_{g-3}=A_0=1$ and if $g>3$, $A_{g-3} \geq B_1(\FF/\F_q) \geq m=3$, we deduce that, if $B_1(\FF/\F_q) \geq 3$ and $g \geq 3$, then $A_{g-2} < h$ and so $\mathcal{E}_{g}$ is true.
			\item If $q=2$ and $g=2$, using assertion (4) of Proposition \ref{prop5}. In fact, $\mathcal{E}_{g}$ is untrue if and only if $h=A_{g-2}=A_1=B_1(\FF/\F_q)$. Since $\mathcal{E}_{g}$ is true if $B_1(\FF/\F_q) \geq 3$, we are left with $h=B_1(\FF/\F_q)=1$ or $2$ and we deduce from Proposition \ref{tab} that there is no solution.
			\item If $q=2$ and $g=2$, using assertion (4) of Proposition \ref{prop5}, $\mathcal{E}_{g}$ is untrue if and only if $h=A_{g-2}=A_0=1$. By Proposition \ref{tab} there are only two function fields $\FF/\F_q$ of genus $2$ such that $h=1$. They are such that $q=2$ and $\FF =\F_2(x,y)$, with
			\begin{itemize}
				\item $y^2+y+(x^5+x^3+1)=0$ and $B_1(\FF/\F_q)=1$, $B_2(\FF/\F_q)=2$.
				\item $y^2+y+(x^3+x^2+1)/(x^3+x+1)=0$ and $B_1(\FF/\F_q)=0$, $B_2(\FF/\F_q)=3$.
			\end{itemize}
			Since $h=1$, all divisors of a given degree $d>0$ are equivalent. In particular, all the divisors of degree $g=2$ are equivalent to any divisor of ${\div}_2^+(\FF/\F_q)$, and therefor they are special.
		\end{enumerate}
	\end{proof}

	The following lemma from H. Niederreiter and C. Xing in \cite[Lemma 6]{nixi} is another characterization of the existence of these divisors. This result is less precise than Proposition \ref{DeggDiv} and the proof is based on the same tools.
	
	\begin{lemma} \label{xing}
		There exists an effective divisor $D$ of $\FF/\F_q$ with $deg(D)=g$ and $dim(D)=1$ if either $B_1(\FF/\F_q) \geq 2$ and $q \geq 3$, or $B_1(\FF/\F_q) \geq 4$ and $q=2$.
	\end{lemma}
	
	\begin{proof}
		The lemma is trivial for $g=0$. if $g=1$, let $D$ be a rational place of $\FF/\F_q$, then $dim(D)=1$. Now let $g \geq 2$. Suppose that $dim(D) \geq 2$ for any positive divisor $D$ with $deg(D)=g$. If $g=2$, then by \cite[Lemma 3(i)]{nixi} we have $A_2=q+h$ and by \cite[Lemma 5]{nixi} we have $A_2 \geq (q+1)h$. Thus $h \leq 1$, which contradicts $h \geq A_1 \geq 2$. \\
		So we may assume $g \geq 3$. Substituting $t=q^{-1/2}$ in the identity in \cite[Lemma 3(ii)]{nixi}, we obtain 
		$$ 2 \sum_{n=0}^{g-2} q^{-n/2} A_n + q^{-(g-1)/2} A_{g-1}= \frac{h-q^{g/2}L(q^{-1/2})}{(q^{1/2}-1)^2 q^{(g-1)/2}}. $$
		Since
		$$ L(q^{-1/2})= \prod_{j=1}^{g} |1-\alpha_j q^{-1/2}|^2 \geq 0,$$
		We infer that 
		\begin{equation}
			2 \sum_{n=0}^{g-2}q^{(g-1-n)/2}A_n + A_{g-1} \leq \frac{h}{(q^{1/2}-1)^2}. \label{FuIn}
		\end{equation}
		\cite[Lemma 3(i)]{nixi} yields $A_g=h+qA_{g-2}$ and \cite[Lemma 5]{nixi} yields $A_g \geq (q+1)h$, and thus $A_{g-2} \geq h$. From \eqref{FuIn} we then get
		$$ 2 q^{1/2} \leq \frac{1}{(q^{1/2}-1)^2}$$
		This inequality is impossible if $q \geq 3$, hence it remains to prove the lemma for $q=2$.\\
		If $q=2$ and $B_1(\FF/\F_q) \geq 4$, then from \eqref{FuIn} we obtain 
		\begin{equation}
			4 A_{g-3} + 2 \sqrt{2} A_{g-2} + A_{g-1} \leq \frac{h}{(\sqrt{2}-1)^2}. \label{FuIn2}
		\end{equation}
		Together with \cite[Lemma 4]{nixi} with $m=3$ and $n=g-1$ this yields 
		\begin{equation}
			A_{g-3} + (3+2 \sqrt{2}) A_{g-2} \leq \frac{h}{(\sqrt{2}-1)^2}. \label{FuIn3}
		\end{equation}
		if we use \cite[Lemma 4]{nixi} with $m=4$ and $n=g-1$ in \eqref{FuIn2}, then we get
		\begin{equation}
			-2 A_{g-3} + (4+2 \sqrt{2}) A_{g-2} \leq \frac{h}{(\sqrt{2}-1)^2}. \label{FuIn4}
		\end{equation}
		By eliminating $A_{g-3}$ from \eqref{FuIn3} and \eqref{FuIn4}, we arrive at
		$$ (10+6 \sqrt{2}) A_{g-2} \leq \frac{3h}{(\sqrt{2}-1)^2}, $$
		and therefore  
		$$ 10+6 \sqrt{2} \leq \frac{3}{(\sqrt{2}-1)^2}, $$
		which is absurd.
		
	\end{proof}
	
	\subsection{Existence of non-special divisors of degree $g-1$} \label{g=1}
	
	In this section, we are interested  by the non-special divisors of degree $g-1$. We begin with the particular case where $g=1$.
	If the genus of $F/\mathbb{F}_q$ is $g = 1$, any divisor of degree $d = g$ is non-special since $d \geq 2g - 1 = 1$ and 
	there exists a non-special divisor of degree $g - 1 = 0$ if and only if the divisor class number $h$ is $> 1$, i.e. $B_1(\FF/\F_q) \geq 2$. 
	So there are exactly 3 function fields of genus 1 which have no non-special divisor of degree $g - 1$. 
	They are the elliptic solutions to the divisor class number one problem (see \cite{mac} and \cite{maqu}):
		
		\begin{align*}
			q &= 2, & y^2 + y + (x^3 + x + 1) &= 0, \\
			q &= 3, & y^2 - (x^3 + 2x + 2) &= 0, \\
			q &= 4, & y^2 + y + (x^3 + a) &= 0, & \text{where } \mathbb{F}_4 = \mathbb{F}_2(a).
		\end{align*}
		
		So, in the rest of this paper, except otherwise stated, we assume that the genus of a function field is $\geq 2$.

	\begin{theorem} \label{g=2}
		Let $\FF/\F_q$ be a function field of genus $g \geq 2$. Then $\mathcal{E}_{g-1}$ is true in the following cases:
		\begin{enumerate}
			\item[\rm(i)] $q \geq 4$.
			\item[\rm(ii)] $g=2$, unless $\FF/\F_q := \F_2(x,y)/\F_2$, with:
			$$ y^2+y+(x^5+x^3+1)=0$$ or $$ y^2+y+(x^4+x+1)/x=0$$
		\end{enumerate}
	\end{theorem}

	\begin{proof}
		Recall that, $A_{g-1}=0$, the existence is clear.
		\begin{enumerate}
			\item $q \geq 4$. By \eqref{a}, for $g\geq 2$ we have:
			$$ A_{g-1} < 2q^{(g-1)/2}A_0+A_{g-1} \leq 2 \sum_{n=0}^{g-2}q^{(g-1-n)/2}A_n+A_{g-1} \leq \frac{h}{(q^{1/2}-1)^2}. $$
			Thus, if $q \geq 4$, we have $A_{g-1}<h$ and the result follows from Proposition \ref{prop5}.
			\item $g=2$. If $A_{g-1}=B_1(\FF/\F_q)<h$, the result follows from Proposition \ref{prop5}. This is the case when $B_1(\FF/\F_q)=0$ and then all divisors of degree $g-1$ are non-special. If $B_1(\FF/\F_q) \geq g+1 = 3$, the result is true by Proposition \ref{prop4}. The remaining cases are $B_1(\FF/\F_q)=1$ or $2$ with $h=B_1(\FF/\F_q)$. By Proposition \ref{tab}, there are two solutions:
			\begin{enumerate}
				\item $B_1(\FF/\F_q)=1$ and $h=1$. There is a unique function field satisfying these conditions. It is $\FF/\F_q := \F_2(x,y)/\F_2$, with: 
				$$ y^2+y+(x^5+x^3+1)=0$$
				Since $h=1$, all divisors of degree $g-1=1$ are equivalent to the place of degree $1$; thus, they are special.
				\item $B_1(\FF/\F_q)=2$ and $h=2$. There is a unique function field satisfying these conditions. It is $\FF/\F_q := \F_2(x,y)/\F_2$, with:
				$$ y^2+y+(x^4+x+1)/x=0$$
				Since the two degree one places are non-equivalent, it follows from Corollary \ref{cor6} that $\mathcal{E}_{g-1}$ is untrue.
			\end{enumerate}
		\end{enumerate}
	\end{proof}

	In the following lemma, the value of $A_{g-1}$ is given in terms of coefficients of the polynomial $L(\FF/\F_q, t)$.
	
	\begin{lemma} \label{Ag-1}
		Let $\FF/\F_q$ be a function field of genus $g$ and let $L(t) = \sum_{i=0}^{2g} a_i t_i$ be the numerator of its Zeta function. Then
		
		$$ A_{g-1} = \frac{1}{q-1} (h- ( a_g + 2 \sum_{i=0}^{g-1} a_i))$$
	\end{lemma}
	
	\begin{proof}
		This is a well-known result
		$$Z(t) = \sum_{m=0}^{+\infty} A_m t^m = \frac{\sum_{i=0}^{2g} a_i t^i}{(1-t)(1-qt)}$$ 
		We deduce that for all $m \geq 0$,
		$$A_m = \sum_{i=0}^{m} \frac{q^{m-i+1}-1}{q-1} a_i $$
		In particular,
		$$(q-1)A_{g-1} = \sum_{i=0}^{g-1} (q^{g-i}-1)a_i$$
		Since $a_i = q^{i-g} a_{2g-i}$, for all $i=0, \ldots , g$, we obtain
		$$(q-1)A_{g-1}= q^g \sum_{i=0}^{g-1} q^{-i}a_i - \sum_{i=0}^{g-1}a_i= q^g \sum_{i=0}^{g-1} q^{-i} q^{i-g} a_{2g-i} - \sum_{i=0}^{g-1}a_i$$
		Hence,
		$$ (q-1)A_{g-1} = \sum_{i=0}^{g-1}(a_{2g-i}-a_i) $$
		Furthermore, we know that $h=L(1)=\sum_{i=0}^{2g} a_i$, therefore,
		$$A_{g-1} = \frac{1}{q-1} \left( h- \left ( a_g+2\sum_{i=0}^{g-1} a_i \right) \right) .$$
	\end{proof}
	
	Using the preceding lemma, Corollary \ref{cor6} and Assertion 2 of Proposition \ref{prop5}, one has 
	
	\begin{corollary} \label{CNS}
	Let $\FF/\F_q$ be a function field of genus $g$ and $L(t) = \sum_{i=0}^{2g} a_i t^i$ be the $L$-polynomial of $\FF/\F_2$. Then
	\begin{itemize}
		\item For $q \geq 3$, $ a_g + 2 \sum_{i=0}^{g-1} a_i \geq 0 $ if and only if $\mathcal{E}_{g-1}$ is true.
		\item For $q = 2$, $ a_g + 2 \sum_{i=0}^{g-1} a_i > 0 $ if and only if $\mathcal{E}_{g-1}$ is true.
	\end{itemize}
	\end{corollary}

	\begin{example}
		The Hermitian function field $\FF/\F_{q^2}$ is such that $\FF = \F_q(x,y)$ with $y^q +y-x^{q+1}=0$. It is a maximal function field of genus $g=\frac{q(q-1)}{2}$ and it is the constant field extension of $\mathbf{G}/\F_q$, where $\mathbf{G} = \F_q(x,y)$, with $y^q+y-x^{q+1}=0$. We can say that $\mathbf{G}/\F_q$ is a "constant field restriction" of $\FF/\F_{q^2}$. All subfields of the Hermitian function field $\FF/\F_{q^2}$ are maximal function fields.
	\end{example}
	
	\begin{corollary}
		If the algebraic function field $\mathbf{G}/\F_q$ is a constant field restriction of a maximal function field $\FF/\F_{q^2} =  \mathbf{G}.\F_{q^2}/\F_{q^2}$, then $\mathbf{G}/\F_q$ contains a non-special divisor of degree $g-1$. 
	\end{corollary}
	
	\subsection{Particular cases : ordinary curves over $\F_2$ and $\F_3$} \label{ex}
	
	The following results treat the particular case of ordinary curves in \cite[Section 4]{bariro}.
	
	Let $\CC/k$ be a genus $g$  (smooth projective absolutely irreducible)  curve over a finite field $k=\F_{p^n}$.  
	Classically, one defines the $p$-rank $\gamma$ of this curve as the integer $0 \leq \gamma \leq g$ such that $\# \Jac(\CC)[p](\overline{k})=p^{\gamma}$. In particular 
	$\CC$ is said to be ordinary if $\gamma=g$.  There is another equivalent characterization in terms of the $L$-polynomial, 
	namely $\gamma=\deg(L(t) \pmod{p})$ (see \cite{mani}). In particular, $\CC$ is ordinary if and only if $p$ does not divide $a_g$. 
	
	\begin{proposition} \label{ordinary}
		Let $\CC$ be an ordinary curve of genus $g>0$ over a finite field $k$ of characteristic $2$.
		There is always a non-special divisor of degree $g-1$ on $\CC$.
	\end{proposition}
	\begin{proof}
		Let $f \in k(\CC)$ such that $df \ne 0$. Developing $f$ in power series at any point of $\CC$, we see that $df$ has only zeros and poles of even multiplicity. 
		Hence there exists a rational divisor of degree $(2g-2)/2=g-1$ such that  $(df)=2 D_0$. It is easy to show that the class of this divisor does not depend on the choice of $f$ 
		and it is called the canonical theta characteristic divisor. In \cite[Prop.3.1]{stvo}, it is shown that there is a bijection between ${\mathcal L}(D_0)$ and the space of  exact regular 
		differentials (i.e. the regular differentials $\omega$ such that $\omega=df$ for $f \in k(\CC)$) . Now by \cite[Prop.8]{serr}, a regular differential $\omega$ is exact if and only 
		if ${\mathcal C}(\omega)=0$ where ${\mathcal C}$ is the Cartier operator.  Moreover by  \cite[Prop.10]{serr}, $\Jac(\CC)$ is ordinary if and only if ${\mathcal C}$ is bijective. 
		So the only exact regular differential is $0$ and $\dim(D_0)=0$. Hence $D_0$ is the divisor we were looking for.
	\end{proof}

	Note, that the previous proof gives a way to explicitly construct a degree $g-1$ divisor of dimension zero.  We will now generalize  Proposition \ref{ordinary} (and Lemma \ref{Ag-1}) but without such an explicit construction.
	
	\begin{lemma} \label{Ag-k}
		Let $\FF/\F_q$ be a function field of genus $g$ and let $L(t)= \sum_{i=0}^{2g}a_it^i$ be the numerator of its Zeta function.
		Then 
		$$A_{g-k}=\frac{1 }{ q-1}\left[ q^{-k+1}\left(h-\sum_{i=0}^{g+k-1}a_i\right)-\sum_{i=0}^{g-k}a_i\right].$$ 
	\end{lemma}
	
	\begin{proof}
	The proof is similar to Lemma \ref{Ag-1} (see \cite[Lemma 3.6]{bariro}).
	\end{proof}
	
	\begin{proposition} \label{ordinary2}
		Let $\CC$ be a curve of genus $g>0$ over a finite field $\F_q$ of characteristic $p$ and of $p$-rank $\gamma$.
		There is always a  degree $\gamma-1$ zero dimension divisor on $\CC$.
	\end{proposition}
	\begin{proof}
		Recall that $h_{n,k}$ the number of classes of divisors of degree $n$ and of dimension $k$, for all $g \geq  k>0$ we get
		$$h=\sum_{i=0}^{\infty} h_{g-k,i}$$
		so $$h_{g-k,0} = h- \sum_{i=1}^{\infty} h_{g-k,i}.$$
		Now $$A_{g-k}= \sum_{i=1}^{\infty} \frac{q^i-1}{q-1} h_{g-k,i}$$
		hence we can write
		$$\sum_{i=1}^{\infty} h_{g-k,i}=  \sum_{i=1}^{\infty} q^i h_{g-k,i}- (q-1) A_{g-k} .$$
		Using the expression of $A_{g-k}$ from Lemma \ref{Ag-k} and then 
		$$h=\sum_{i=0}^{g+k-1} a_i + \sum_{i=g+k}^{2g}  a_i \equiv \sum_{i=0}^{\gamma} a_i \pmod{p}$$
		we get for $k=g-\gamma+1$
		\begin{eqnarray*}
			h_{g-k,0} &=& h (1+q^{-k+1})  - q^{-k+1} \sum_{i=0}^{g+k-1} a_i - \sum_{i=0}^{g-k} a_i - \sum_{i=1}^{\infty} q^i h_{g-k,i} \\
			&=&  \sum_{i=0}^{2 g} a_i+q^{-k+1}  \sum_{i=g+k}^{2g} a_i- \sum_{i=0}^{g-k} a_i- \sum_{i=1}^{\infty} q^i h_{g-1,i} \\
			&=&  \sum_{i=g-k+1}^{ 2g } a_i+q^{-k+1}  \sum_{i=g+k}^{2g} a_i- \sum_{i=1}^{\infty} q^i h_{g-1,i} \\
			& \equiv &  \sum_{i=g-k+1}^{ \gamma } a_i \pmod{p} \\
			& \equiv & a_{\gamma} \not \equiv 0 \pmod{p}.
		\end{eqnarray*} 
		Hence $h_{\gamma-1,0}$ is not zero and hence is positive.
	\end{proof}
	
	\begin{remark}	
		Note that this proposition is interesting only in the case where $q=2$ and  $\gamma=g-k$ with $k \leq 3$ or $q=3$ and  $\gamma=g$.
	\end{remark}
	
	\begin{corollary} \label{ordinary3}
	Let $\CC$ be an ordinary curve of genus $g>0$ over a finite field $\F_q$. There is always a non-special divisor of degree $g-1$ on $\CC$.	
	\end{corollary}
	
	\subsection{Particular case : Asymptotically exact towers}
	
	In this section we adapt the results in \cite[Section 5.2]{bariro} to prove the existence of non-special divisors of degree $g-1$ in asymptotically exact towers. 
	
	First, let us recall 
	the notion of asymptotically exact sequence
	of algebraic function fields introduced in \cite{tsfa}.

	\begin{definition}
		Consider a sequence ${\mathcal F}/\F_q=(\FF_k/\F_q)_{k\geq 1}$  
		of algebraic function fields $\FF_k/\F_q$ defined over $\F_q$ of genus $g_k$. 
		We suppose that  the sequence of genus $g_k$ is an increasing 
		sequence growing to infinity. 
		The sequence ${\mathcal F}/\F_q$ is called \emph{asymptotically 
			exact} if for all $m \geq 1$ the following limit exists:  
		$$\beta_m({\mathcal F}/\F_q)= \lim_{g_k \rightarrow \infty} \frac{B_m(\FF_k/\F_q)}{g_k} $$  
		where $B_m(\FF_k/\F_q)$ is the number of places of degree $m$
		on $\FF_k/\F_q$. 
	\end{definition}
	
	Now, let us recall the following two results used by I. Shparlinski, 
	M. Tsfasman and S. Vladut in \cite{shtsvl}. These results follow easily from Corollary 2 and Theorem 6 of \cite{tsfa}.
	
	\begin{lemma}\label{lem1}
		Let ${\mathcal F}/\F_q=(\FF_k/\F_q)_{k \geq 1}$ be an asymptotically exact
		sequence 
		of algebraic function fields defined over $\F_q$ and 
		$h_k$ be the class number of $\FF_k/\F_q$. 
		Then $$log_qh_k=g_k\left( 1+\sum_{m=1}^{\infty}\beta_m \cdot log_q\frac{q^m}{q^m-1} \right)+o(g_k)$$
	\end{lemma}

	\begin{lemma}\label{lem2}
		Let $A_{a_k}$ be the number of effective divisors of degree $a_k$ on
		$\FF_k/\F_q$.
		If $$a_k\geq g_k\left(\sum_{m=1}^{\infty}\frac{m\beta_m}{q^m-1}
		\right)+o(g_k)$$ 
		then
		$$log_qA_{a_k}=a_k+g_k \cdot \sum_{m=1}^{\infty}\beta_m \cdot log_q\frac{q^m}{q^m-1}+o(g_k).$$
	\end{lemma}
	
	These asymptotical properties were established in \cite{tsfa} and \cite{tsvl} in order to estimate the class number $h$ 
	of algebraic function fields of genus $g$ defined over $\F_q$ and also in order to estimate their number of classes of effective divisors of degree $m\leq g-1$. Namely, I. Shparlinski, M. Tsfasman and S. Vladut used in \cite{shtsvl} the inequality $2A_{g_k(1-\epsilon)}<h_k$ where $0< \epsilon <\frac{1}{2}$ and $k$ big enough, under the hypothesis of Lemma \ref{lem1}. In the same spirit, we give here a particular case of their result in the following proposition.
	\begin{proposition}\label{ExiAsy}
		Let ${\mathcal F}/\F_q=(\FF_k/\F_q)_{k\geq 1}$ be an 
		asymptotically exact sequence of algebraic function field defined over $\F_q$.
		Then, there exists an integer $k_0$ such that for any integer $k\geq k_0$, we get:  
		$$A_{g_k-1}<h_k$$ 
		and there is a non-special divisor of degree $g-1$ in $\FF_k/\F_q$. 
	\end{proposition}
	
	\begin{proof}
		The total number of linear equivalence classes of an arbitrary 
		degree equals to the divisor class number $h_k$ of $\FF_k/\F_q$, 
		which is given by Lemma \ref{lem1}. Moreover, for $g_k$ sufficiently large, 
		we have:
		$$\sum_{m=1}^{\infty}\frac{m\beta_m}{q^m-1} \leq  \frac{1}{\sqrt{q}+1} \sum_{m=1}^{\infty}\frac{m\beta_m}{q^{\frac{m}{2}}-1}<\frac{1}{2}$$
		since $q\geq 2$ and $ \sum_{m=1}^{\infty} \frac{m\beta_m}{q^{\frac{m}{2}}-1}\leq 1$ 
		by Corollary 1 of \cite{tsfa}. As $\frac{1}{g_k}<\frac{1}{2}$, one has
		$$ g_k(1-\frac{1}{g_k}) \geq g_k\left(\sum_{m=1}^{\infty}\frac{m\beta_m}{q^m-1}
		\right)+o(g_k)$$ 
		for $k$ big enough.
		Therefore, we can apply Lemma \ref{lem2} and compare $\log_q A_{ g_k(1-1/g_k) }$ with $\log_q h_k$ given by Lemma \ref{lem1}. Hence, there exists an integer $k_0$ such that for $k \geq k_0$,   
		$A_{ g_k-1 }<h_k$. We conclude by Proposition \ref{prop5}. 
	\end{proof}
	
	\subsection{Particular case : curves of defect $k$ over $\F_2$ and $\F_3$} \label{NR}
	
	In this section, we will focus on curves over $\F_2$ and $\F_3$ of genus $g \geq 3$. The existence of non-special divisors of degree $g-1$ is assured for $q \geq 4$, moreover, the cases of curves of genus $g = 1$, $2$ were studied in the introduction of section \ref{g=1} and Theorem \ref{g=2}. The goal is to give some examples of function fields that contain these divisors.
	
	We consider for $r \geq 1$ the number
	$$N_r := N(\FF_r) = \#\{ P \in \bP_k(\FF/\F_q); deg(P)=1 \}$$
	where $\FF_r = \FF \F_{q^r}$ is the constant field extension of $\FF/\F_q$ of degree $r$. Let us remind the equation from \cite{stic}*{Corollary 5.1.16}, for all $r \geq 1$,
	\begin{equation} 
		N_r = q^r+1-\sum_{i=1}^{2g} \alpha_i^r = q^r+1 - \sum_{i=1}^{g} 2q^{r/2} cos(\pi r \phi_i) \label{NiFoncAl} 
	\end{equation}
	
	where $(\alpha_1,\ldots,\alpha_{2g}) = (q^{1/2}e^{i\pi \phi_1},\ldots,q^{1/2}e^{i\pi \phi_g},q^{1/2}e^{-i\pi \phi_1},\ldots,q^{1/2}e^{-i\pi \phi_g})$ are the reciprocals of the roots of $L(t)$ with $\phi_i \in [0,1]$. In particular, since $N_1=N(\FF)$, we have
	$$ N(\FF) = q +1 - \sum_{i=1}^{2g} \alpha_i. $$
	
	\begin{proposition} \label{RelRec} \cite{stic}*{Corollary 5.1.17}
		Let $L(t) = \sum_{i=0}^{2g} a_i t^i$ be the $L$-polynomial of $\FF/\F_q$, and $S_r := N_r-(q^r+1)$. Then we have:
		\begin{enumerate}[label={(\alph*)}]
			\item $L'(t)/L(t) = \sum_{r=1}^{\infty} S_r t^{r-1}$.
			\item $a_0 = 1$ and 
			\begin{equation}
				i a_i = S_i a_0 + S_{i-1} a_1 + \cdots + S_1 a_{i-1} \label{RelRecEg}
			\end{equation}
			for $i=1, \ldots, g$. \\
		\end{enumerate}
		Given $N_1, \ldots, N_g$ and using $a_{2g-i}=q^{g-i}a_i$, we can determine $L(t)$ from \eqref{RelRecEg}.
	\end{proposition}
	
	Let $\FF$ be an algebraic function field over $\F_2$ or $\F_3$ of genus $g$, and $k$ its defect, \textit{i.e.}
	$$ |N_1(\FF/\F_2)-(q+1)| =  g[2\sqrt{q}]-k $$
	where $[x]$ denotes the largest integer $\leq x$.
	Let 
	$$L(t)=\sum_{i=0}^{2g} a_i t^i=\prod_{i=1}^{g} [(1-\alpha_i t)(1-\overline{\alpha_i} t)]$$
	be the numerator of the Zeta function of $\FF$. Recall that with the previous conditions, one has
	$$ |N_1(\FF/F_q)-(q+1)|=|-\sum_{i=1}^{2g} \alpha_i| = g[2\sqrt{q}]-k $$
	and then 
	\begin{equation}
		k = g[2\sqrt{q}] - |\sum_{i=1}^{2g} \alpha_i|. \label{defect}
	\end{equation}
	
	 We know that if $N_1 \geq g+1$, there exists a non-special divisor of degree $g-1$. Moreover we know all the curves which contain this kind of divisors over $\F_2$ for $g=1$ and $2$. Since for defect 2 curves over $\F_2$ with $g \geq 3$ one has 
	$$|N_1-3|=2g-2\ \Rightarrow N_1 = 2g+1 \Rightarrow N_1 \geq g+1$$
	the existence of these divisors is obvious and we do not need to use the coefficients $a_n$ for this purpose. Nevertheless, for a defect $k \geq 3$ curves over $\F_2$ and $\F_3$ ($q = 2$ or $3$, hence $[2\sqrt{q}]= q$ ), one has
	$$|N_1-(q+1)|=qg-k\ \Rightarrow N_1 = -qg+k+q+1\ \hbox{or}\ N_1=qg-k+q+1 $$
	the goal of this section is to prove the existence of non-special divisor of degree $g-1$ for curves that satisfy
	\begin{equation}
		0 \leq N_1 = -qg + k + q + 1 \leq g\ \hbox{or}\  0 \leq N_1 =  qg -k + q + 1 \leq g. \label{CEx}
	\end{equation}
	
	the following results generalize the cases (a) and (d) of \cite[Theorem 2.5.1]{serre} by J.P Serre and will help us determine the sign of $ a_g + 2 \sum_{i=0}^{g-1} a_i$ (in order to apply \ref{CNS}). 
	
	\begin{lemma} \label{RelDefTr}
		Let $\alpha$ be a totally positive algebraic integer and \linebreak[4] $\mathsf{k}(\alpha)= Tr(\alpha) - d(\alpha)$ (recall that $Tr(\alpha)$ is the sum of the conjugates of $\alpha$, it is totally positive if all its conjugates are real $> 0$, that $d(\alpha)$ is the degree of its minimal polynomial and $Tr(\alpha) \geq deg(\alpha)$ (see \cite[Remark 2.2.3]{serre})). One has
		\begin{itemize}
			\item If $\mathsf{k}(\alpha) = 0$, then $\alpha =1$.
			\item If $d(\alpha) = 1$, then $\alpha = \mathsf{k}(\alpha)+1$.
			\item If $d(\alpha) = 2$, then $\alpha = \frac{\mathsf{k}(\alpha)+2 \pm \sqrt{(\mathsf{k}(\alpha)+2)^2-4n}}{2}$ with $\frac{(\mathsf{k}(\alpha)+2)^2}{4} > n$.
		\end{itemize}
	\end{lemma}
	
	\begin{proof}
		\begin{itemize}
			\item For the case $\mathsf{k}(\alpha) = 0$ see \cite[Corollary 2.2.4]{serre}.
			\item If $d(\alpha) = 1 $, then the minimal polynomial of $\alpha$ is $x-Tr(\alpha)$ and $\alpha = Tr(\alpha)$, we conclude that $\mathsf{k}(\alpha)= Tr(\alpha) - 1 = \alpha -1$. 
			\item If $d(\alpha) = 2 $, then the minimal polynomial of $\alpha$ is \linebreak[4 ]$x^2-Tr(\alpha)+n = x^2 -(\mathsf{k}(\alpha)+2)+n$ with real positive roots, namely $\frac{\mathsf{k}(\alpha)+2 \pm \sqrt{(\mathsf{k}(\alpha)+2)^2-4n}}{2}$ only if $\frac{(\mathsf{k}(\alpha)+2)^2}{4} > n$.
		\end{itemize}
	\end{proof}

	\begin{theorem} \label{SCaseN}
		For a curve over $\F_q$ such that $|N_1(\FF/\F_q)-(q+1)| = g[2\sqrt{q}]-k$ with $g \geq 2$, one has 
		
		\begin{enumerate}[label={(\alph*)}]
			\item For $ 3 \leq k \leq 2 \cdot [2\sqrt{q}]$, if there exists $i \in \{1, \ldots, g \}$ such that $ \alpha_i+\overline{\alpha_i} = [2\sqrt{q}] - k $, then we can reorganize the tuple $(\alpha_1+\overline{\alpha_1},\ldots, \alpha_g+\overline{\alpha_g})$ to get
			$$(\alpha_1+\overline{\alpha_1},\ldots, \alpha_g+\overline{\alpha_g}) = \pm([2\sqrt{q}], \ldots,[2\sqrt{q}],[2\sqrt{q}]-k)$$ 
			\item If  there exists $i, j \in \{1, \ldots, g \}$ such that  $[2\sqrt{q}]+1-\alpha_i-\overline{\alpha_i}$ and \linebreak[4] $[2\sqrt{q}]+1-\alpha_j-\overline{\alpha_j} $ are conjugated with $d([2\sqrt{q}]+1-\alpha_i-\overline{\alpha_i}) = 2$ and
			$$\alpha_i+\overline{\alpha_i}+\alpha_j+\overline{\alpha_j} = 2[2\sqrt{q}] - k  $$
			then we can reorganize the tuple $(\alpha_1+\overline{\alpha_1},\ldots, \alpha_i+\overline{\alpha_i} ,\ldots, \alpha_g+\overline{\alpha_g})$ to get
			
		    \begin{align*}
		    	(\alpha_1+\overline{\alpha_1}, \ldots, \alpha_g+\overline{\alpha_g}) 
		    	= \pm & \Big([2\sqrt{q}], \ldots, [2\sqrt{q}], \\
		    	&[2\sqrt{q}] + 1 - \frac{k+2 + \Delta}{2}, [2\sqrt{q}] + 1 - \frac{k+2 - \Delta}{2} \Big)
		    \end{align*}

			with $\Delta = \sqrt{(k+2)^2-4n}$ and $\frac{(k+2)^2}{4} > n$. \\
			This holds for $-2(2\sqrt{q}-[2\sqrt{q}]) \leq k \pm \Delta \leq 4\sqrt{q}+2[2\sqrt{q}] $. 
		\end{enumerate}
		
	\end{theorem}
	
	\begin{proof}
		It is enough to prove the proposition in the case \linebreak[4] $N_1(\FF/\F_q)-(q+1) = g[2\sqrt{q}]-k$ which means $k = g[2\sqrt{q}] - \sum_{i=1}^{2g} \alpha_i$ (by \eqref{defect}).\\
		Let $ \kappa : \Z[X] \rightarrow \Z $ be the map defined by 
		$$\kappa(b_0 X^n - b_1 X^{n-1}+ \ldots + b_n) = b_1 - n $$
		and $P \in \Z[X]$ be the polynomial 
		$$P(X) =  X^g - a_1 X^{g-1} + \ldots + a_g = \prod_{i=1}^{g} (X-[2\sqrt{q}]-1+\alpha_i+\overline{\alpha_i}) $$
		its roots are real positive since $[2\sqrt{q}]+1 \geq \alpha_i+\overline{\alpha_i} = 2\sqrt{q}\cdot cos(\phi_i)$ with $\phi_i$ the argument of $\alpha_i$.
		then 
		$$\kappa(P(X)) = \sum_{i=1}^{g} ([2\sqrt{q}]+1-\alpha_i-\overline{\alpha_i}) - g = g [2\sqrt{q}]+g- \sum_{i=1}^{g}(\alpha_i+\overline{\alpha_i}) - g =k  $$
		Now, let $P(X) = \prod_{\lambda=1}^{r} Q_{\lambda}(X)= \prod_{\lambda=1}^{r}(X^{deg(Q_\lambda)}-a_{1,\lambda}X^{deg(Q_\lambda)-1}+\ldots)$ be the decomposition of $P$ in a product of irreducible polynomials. We have 
		$$ a_1  =  \sum_{\lambda=1}^{r} a_{1,\lambda}  $$
		and 
		$$ a_1 - g =  \sum_{\lambda=1}^{r} a_{1,\lambda} - \sum_{\lambda = 1}^{r} deg(Q_\lambda)$$
		thus
		\begin{equation}
			\kappa(P(X)) = \sum_{\lambda=1}^{r} \kappa(Q_\lambda(X)) = k. \label{DefPDefC}
		\end{equation}
		
		\begin{enumerate}[label={(\alph*)}]
			\item Let $i \in \{1, \ldots, g \}$ such that $ \alpha_i+\overline{\alpha_i} = [2\sqrt{q}] - k $, then \linebreak[4] $[2\sqrt{q}]+1-\alpha_i-\overline{\alpha_i} \in \Z$ which means that there exists $\lambda' \in \{1, \ldots, r \}$ such that $Q_{\lambda'}(X) = X - [2\sqrt{q}]-1+\alpha_i+\overline{\alpha_i} $ then $d([2\sqrt{q}]+1-\alpha_i-\overline{\alpha_i}) = 1$ and $Tr([2\sqrt{q}]+1-\alpha_i-\overline{\alpha_i})= [2\sqrt{q}]+1-\alpha_i-\overline{\alpha_i}$. By \eqref{DefPDefC}
			\begin{eqnarray*}
				\kappa(P(X)) & = & \sum_{\lambda=1}^{r} \kappa(Q_\lambda(X))\\
				& = & \sum\limits_{\substack{\lambda =1 \\ \lambda \neq \lambda'}}^{r} \kappa(Q_\lambda(X)) + \kappa(Q_{\lambda'(X)}) \\
				& = & \sum\limits_{\substack{\lambda =1 \\ \lambda \neq \lambda'}}^{r} \kappa(Q_\lambda(X)) +[2\sqrt{q}]+1-\alpha_i-\overline{\alpha_i} -1 \\
				& = & \sum\limits_{\substack{\lambda =1 \\ \lambda \neq \lambda'}}^{r} \kappa(Q_\lambda(X)) +[2\sqrt{q}]- [2\sqrt{q}] + k \\
				& = & \sum\limits_{\substack{\lambda =1 \\ \lambda \neq \lambda'}}^{r} \kappa(Q_\lambda(X)) + k = k
			\end{eqnarray*} 
			Notice that if $[2\sqrt{q}]+1-\alpha_j-\overline{\alpha_j}$ is a root of $Q_\lambda$ ($\lambda \neq \lambda'$) then with the notation of Lemma \ref{RelDefTr} one has 
			
			\begin{align*}
				\kappa(Q_\lambda(X)) &= Tr([2\sqrt{q}] + 1 - \alpha_j - \overline{\alpha_j}) 
				- d([2\sqrt{q}] + 1 - \alpha_j - \overline{\alpha_j}) \\
				&= \mathsf{k}([2\sqrt{q}] + 1 - \alpha_j - \overline{\alpha_j}) \geq 0
			\end{align*}

			we conclude that $\kappa( Q_\lambda(X)) = \mathsf{k}([2\sqrt{q}]+1-\alpha_j-\overline{\alpha_j}) = 0$ for $\lambda \neq \lambda'$. By Lemma \ref{RelDefTr}, one has $[2\sqrt{q}]+1-\alpha_j-\overline{\alpha_j} = 1$ thus $[2\sqrt{q}]=\alpha_j+\overline{\alpha_j}$ for $j \neq i$ and by assumption $ \alpha_i+\overline{\alpha_i} = [2\sqrt{q}] - k $. Finally, the argument $\theta_i$ with $2\sqrt{q} \cdot cos(\theta_i)=\alpha_i+\overline{\alpha_i}=[2\sqrt{q}] - k$ exists if $0 \leq k \leq 2 \cdot [2\sqrt{q}]$. We are interested, here, by $3 \leq k \leq 2 \cdot [2\sqrt{q}]$ since the cases $0$, $1$ et $2$ were studied in \cite{serre}*{Theorem 2.5.1}.
			\item Let $\beta_i=[2\sqrt{q}]+1-\alpha_i-\overline{\alpha_i}$ and $\beta_j=[2\sqrt{q}]+1-\alpha_j-\overline{\alpha_j}$.\\
			There exists $\lambda' \in \{1, \ldots, r \}$ such that
			$$Q_{\lambda'}(X) = X^2 - a_{1,\lambda'} X + a_{2,\lambda'}$$ with  $$a_{1,\lambda'}=Tr(\beta_i)=Tr(\beta_j)=2[2\sqrt{q}]+2-\alpha_i-\overline{\alpha_i}-\alpha_j-\overline{\alpha_j}   $$
			By \eqref{DefPDefC}, one has
			\begin{eqnarray*}
				\kappa(P(X)) & = & \sum_{\lambda=1}^{r} \kappa(Q_\lambda(X))\\
				& = & \sum\limits_{\substack{\lambda =1 \\ \lambda \neq \lambda'}}^{r} \kappa(Q_\lambda(X)) + \kappa(Q_{\lambda'(X)}) \\
				& = & \sum\limits_{\substack{\lambda =1 \\ \lambda \neq \lambda'}}^{r} \kappa(Q_\lambda(X)) +2[2\sqrt{q}]+2-\alpha_i-\overline{\alpha_i}-\alpha_j-\overline{\alpha_j} -2 \\
				& = & \sum\limits_{\substack{\lambda =1 \\ \lambda \neq \lambda'}}^{r} \kappa(Q_\lambda(X)) +2[2\sqrt{q}]-2[2\sqrt{q}] + k \\
				& = & \sum\limits_{\substack{\lambda =1 \\ \lambda \neq \lambda'}}^{r} \kappa(Q_\lambda(X)) + k = k
			\end{eqnarray*} 
			As in a), we conclude that $\kappa( Q_\lambda(X)) = \mathsf{k}([2\sqrt{q}]+1-\alpha_j-\overline{\alpha_j}) = 0$ for $\lambda \neq \lambda'$ and since $d(\beta_i)=d(\beta_j) = 2$, by Lemma \ref{RelDefTr} one has
			$$\beta_i = \frac{\mathsf{k}(\beta_i)+2 - \sqrt{(\mathsf{k}(\beta_i)+2)^2-4n}}{2} = \frac{k+2 + \sqrt{(k+2)^2-4n}}{2} $$
			and 
			$$\beta_j = \frac{\mathsf{k}(\beta_j)+2 - \sqrt{(\mathsf{k}(\beta_j)+2)^2-4n}}{2} = \frac{k+2 - \sqrt{(k+2)^2-4n}}{2} $$
			thus $\alpha_i+\overline{\alpha_i}= [2\sqrt{q}]+1-\frac{k+2 + \Delta}{2}$ and $\alpha_j+\overline{\alpha_j}= [2\sqrt{q}]+1-\frac{k+2 - \Delta}{2}$.\\
			Finally, the arguments $\theta_i$, $\theta_j$ with $2\sqrt{q} \cdot cos(\theta_i)=\alpha_i+\overline{\alpha_i}$ and \linebreak[4] $2\sqrt{q} \cdot cos(\theta_j)=\alpha_j+\overline{\alpha_j}$ exist if
			$$ -1 \leq \frac{[2\sqrt{q}]+1-\frac{k+2 \pm \Delta}{2}}{2\sqrt{2}} \leq 1 $$
			which means 
			$$ -2(2\sqrt{q}-[2\sqrt{q}]) \leq k \pm \Delta \leq 4\sqrt{q}+2[2\sqrt{q}].$$
		\end{enumerate}

	\end{proof}
	
	\begin{remark}
		The conditions of Theorem \ref{SCaseN} are too constraining but they describe a possibly infinitely many curves if we refer to the following theorem (\cite{serre}*{Theorem 2.4.1}). 
	\end{remark}
	
	\begin{theorem} \label{ExiCur}
		Let $\{ \alpha_1+\overline{\alpha_1},\ldots, \alpha_g+\overline{\alpha_g} \}$ be a set of algebraic integers. Suppose that the polynomial $\prod_{i=1}^{g}(X-\alpha_i-\overline{\alpha_i})$ can be factored as $P_1 \cdot P_2$ such that $P_1$ and $P_2$ are monic, non constant polynomials in $\Z[X]$ and their resultant is equal to $1$ or $-1$. Then the $(\alpha_i)_{1 \leq i \leq g}$ can not come from a curve. 
		
	\end{theorem}
	
	Now, we have enough information to calculate the sum  $ a_g + 2 \sum_{i=0}^{g-1} a_i $ using \eqref{NiFoncAl} and Proposition \ref{RelRec}, in order to apply Corollary \ref{CNS}. Therefor, under the conditions of \eqref{CEx}, Theorem \ref{SCaseN} case a) and Theorem \ref{ExiCur}, we can build the following table: \\

	\begin{table}[htb]
		\centering
		\hfill
		\begin{tabular}{|C{0.5cm}|C{0.5cm}|C{0.5cm}|C{1.5cm}|}
			\hline q & g & k &   $\mathcal{E}_{g-1}$ \\
			\hline \hline 2 & 3 &  3 & True \\
			\hline 2 & 3 & 4 &  True \\
			\hline
		\end{tabular}
		\hfill
		\begin{tabular}{|C{0.5cm}|C{0.5cm}|C{0.5cm}|C{1.5cm}|}
			\hline q & g & k &   $\mathcal{E}_{g-1}$ \\
			\hline\hline 3 & 3 & 5 &  True \\
			\hline 3 & 3 & 6 &  True \\
			\hline
		\end{tabular}
		\hfill\null
		\caption{}
		\label{Tab}
	\end{table}

	and under the conditions of \eqref{CEx}, Theorem \ref{SCaseN} case b) and Theorem \ref{ExiCur}, we can build the following table: 
	
	\newpage
	
	\begin{table}[htb]
		\centering
		\hfill
		\begin{tabular}{|C{0.5cm}|C{0.5cm}|C{0.5cm}|C{0.5cm}|C{1.5cm}|}
			\hline q & g & k &  n &  $\mathcal{E}_{g-1}$ \\
			\hline\hline 2 & 3 & 3 & 1 & False  \\
			\hline 2 & 3 & 3 & 2 & True  \\
			\hline 2 & 3 & 3 & 4 & True \\
			\hline 2 & 3 & 3 & 6 & False  \\
			\hline 2 & 3 & 4 & 1 & True  \\
			\hline 2 & 3 & 4 & 2 & False  \\
			\hline 2 & 3 & 4 & 3 & False  \\
			\hline 2 & 3 & 4 & 5 & True  \\
			\hline 2 & 3 & 4 & 7 & True  \\
			\hline 2 & 3 & 4 & 8 & True  \\
			\hline 2 & 3 & 4 & 9 & True  \\
			\hline 2 & 3 & 5 & 8 & True  \\
			\hline 2 & 4 & 5 & 8 & True  \\
			\hline 2 & 3 & 5 & 9 & True  \\
			\hline 2 & 4 & 5 & 9 & True  \\
			\hline 2 & 3 & 5 & 10 & True  \\
			\hline 2 & 4 & 5 & 10 & True  \\
			\hline 2 & 3 & 5 & 11 & True  \\
			\hline 2 & 4 & 5 & 11 & True  \\
			\hline 2 & 3 & 5 & 12 & True  \\
			\hline 2 & 4 & 5 & 12 & True  \\
			\hline 2 & 3 & 6 & 13 & True  \\
			\hline 2 & 4 & 6 & 13 & True  \\
			\hline 2 & 3 & 6 & 14 & True  \\
			\hline 2 & 4 & 6 & 14 & True  \\
			\hline 2 & 3 & 6 & 15 & True  \\
			\hline 2 & 4 & 6 & 15 & True  \\
			\hline 2 & 3 & 6 & 16 & True  \\
			\hline 2 & 4 & 6 & 16 & True   \\
			\hline 2 & 3 & 7 & 19 & True  \\
			\hline 2 & 4 & 7 & 19 & True  \\
			\hline 2 & 5 & 7 & 19 & True  \\
			\hline 2 & 3 & 7 & 20 & True  \\
			\hline 2 & 4 & 7 & 20 & True  \\
			\hline 2 & 5 & 7 & 20 & True  \\
			\hline 2 & 3 & 8 & 25 & True  \\
			\hline 2 & 4 & 8 & 25 & True  \\
			\hline 
		\end{tabular}
		\hfill
		\begin{tabular}{|C{0.5cm}|C{0.5cm}|C{0.5cm}|C{0.5cm}|C{1.5cm}|}
			\hline q & g & k &  n &  $\mathcal{E}_{g-1}$ \\
			\hline\hline 2 & 5 & 8 & 25 & True  \\
			\hline 3 & 3 & 5 & 4 & True  \\
			\hline 3 & 3 & 5 & 6 & True \\
			\hline 3 & 3 & 5 & 8 & True \\
			\hline 3 & 3 & 5 & 9 & False \\
			\hline 3 & 3 & 5 & 10 & False \\
			\hline 3 & 3 & 5 & 11 & False \\
			\hline 3 & 3 & 5 & 12 & False \\
			\hline 3 & 3 & 6 & 4 & False \\
			\hline 3 & 3 & 6 & 5 & False \\
			\hline 3 & 3 & 6 & 7 & True \\
			\hline 3 & 3 & 6 & 9 & True \\
			\hline 3 & 3 & 6 & 10 & True \\
			\hline 3 & 3 & 6 & 11 & True \\
			\hline 3 & 3 & 6 & 12 & True \\
			\hline 3 & 3 & 6 & 13 & True \\
			\hline 3 & 3 & 6 & 14 & True \\
			\hline 3 & 3 & 6 & 15 & True \\
			\hline 3 & 3 & 6 & 16 & True \\
			\hline 3 & 3 & 7 & 12 & False \\
			\hline 3 & 3 & 7 & 13 & False \\
			\hline 3 & 3 & 7 & 14 & False \\
			\hline 3 & 3 & 7 & 15 & False \\
			\hline 3 & 3 & 7 & 16 & True \\
			\hline 3 & 3 & 7 & 17 & True \\
			\hline 3 & 3 & 7 & 18 & True \\
			\hline 3 & 3 & 7 & 19 & True \\
			\hline 3 & 3 & 7 & 20 & True \\
			\hline 3 & 3 & 8 & 19 & True \\
			\hline 3 & 4 & 8 & 19 & True \\
			\hline 3 & 3 & 8 & 20 & False \\
			\hline 3 & 4 & 8 & 20 & False \\
			\hline 3 & 3 & 8 & 21 & True \\
			\hline 3 & 4 & 8 & 21 & True \\
			\hline 3 & 3 & 8 & 22 & True \\
			\hline 3 & 4 & 8 & 22 & True \\
			\hline 3 & 3 & 8 & 23 & True \\
			\hline
		\end{tabular}
		\hfill\null
	\end{table}
	
	\begin{table}[htb]
		\centering
		\hfill
		\begin{tabular}{|C{0.5cm}|C{0.5cm}|C{0.5cm}|C{0.5cm}|C{1.5cm}|}
			\hline q & g & k &  n &  $\mathcal{E}_{g-1}$ \\
			\hline \hline 3 & 4 & 8 & 23 & True \\
			\hline 3 & 3 & 8 & 24 & True \\
			\hline 3 & 4 & 8 & 24 & True \\
			\hline 3 & 3 & 8 & 25 & True \\
			\hline 3 & 4 & 8 & 25 & True \\
			\hline 3 & 3 & 9 & 27 & True \\
			\hline 3 & 4 & 9 & 27 & True \\
			\hline 3 & 3 & 9 & 28 & False \\
			\hline 3 & 4 & 9 & 28 & False \\
			\hline 3 & 3 & 9 & 29 & True \\
			\hline 3 & 4 & 9 & 29 & True \\
			\hline 3 & 3 & 9 & 30 & True \\
			\hline 3 & 4 & 9 & 30 & True \\
			\hline
		\end{tabular}
		\hfill
		\begin{tabular}{|C{0.5cm}|C{0.5cm}|C{0.5cm}|C{0.5cm}|C{1.5cm}|}
			\hline q & g & k &  n &  $\mathcal{E}_{g-1}$ \\
			\hline \hline 3 & 3 & 10 & 34 & True \\
			\hline 3 & 4 & 10 & 34 & True \\
			\hline 3 & 3 & 10 & 35 & True \\
			\hline 3 & 4 & 10 & 35 & True \\
			\hline 3 & 3 & 10 & 36 & True \\
			\hline 3 & 4 & 10 & 36 & True \\
			\hline 3 & 3 & 11 & 42 & True \\
			\hline 3 & 4 & 11 & 42 & True \\
			\hline 3 & 5 & 11 & 42 & True \\
			\hline 3 & 3 & 12 & 49 & True \\
			\hline 3 & 4 & 12 & 49 & True \\
			\hline 3 & 5 & 12 & 49 & True \\
			\hline
		\end{tabular}
		\hfill\null
		\caption{}
		\label{Tab}
	\end{table}

	where $g$ is the genus, $k$ the defect and $n$ the integer defined in Lemma \ref{RelDefTr}. \\
	
	 In the table provided, we present the signs of the sum $ a_g + 2 \sum_{i=0}^{g-1} a_i$ for curves with defect \( k > 2 \). These results were obtained through direct computation, without a formal proof. However, it is worth noting that these signs can be theoretically demonstrated using the same method developed in \cite{ko} for curves with defect 2. That method relies on an explicit form of the coefficients of the L-polynomial, which provides a rigorous framework for extending these results to higher defect values, should one wish to pursue a theoretical proof.

	\section{Construction of non-special divisors of degree $g$ and $g-1$}
	
	\subsection{Example on Kummer extension and hermitian curves}
	
	This section presents the work of E. Camps Moreno, H. H. Lopez and G. L. Matthews in \cite{caloma} which is based on the Weierstrass semigroup to find the explicit form of non-special divisors of degree $g$ and $g-1$ on Kummer extensions and Hermitian curves.
	
	For a curve $\CC$ defined by $f(y) = g(x)$ over a finit field $\F_q$, we denote $P_{ab}$ a point on $\CC$ corresponding to $x=a$ and $y=b$. If $\CC$ has a unique point at infinity, we note it by $P_{\infty}$.\\
	
	We say that a divisor $A$ is supported by a point $P \in \CC(\F_q)$ if and only if $v_p(A) \neq 0$.
	
	Given $m$ distinct rational points $P_1, \dots, P_m$ on a curve $\CC$, the Weierstrass semigroup of the $m$-tuple is 
	$$H(P_1, \dots, P_m) := \{ \alpha \in \mathbb{N}^m : \exists f \in \CC(\F_q))\ with\ (f)_{\infty} = \sum_{i=1}^{m} \alpha_i P_i \}$$ 
	
	Equivalently, $\alpha \in H(P_1, \dots, P_m)$ if and only if 
	$$ dim(\sum_{i=1}^m \alpha_i P_i) = dim(\sum_{i=1}^m \alpha_i P_i - P_j) + 1$$
	
	For all $j$, $ 1 \leq j \leq m$. The Weierstrass gap of $(P_1, \dots, P_m)$ is
	$$ G(P_1, \dots, P_m) := \mathbb{N}^m \setminus H(P_1, \dots, P_m)$$
	
	The multiplicity of the semigroup $H(P)$, where $P \in \CC(\F_q)$ is 
	
	$$ \gamma (H(P)) := min\{a : a \in H(P) \setminus \{0\}\}.$$
	
	We define a partial order $ \preceq$  on $\mathbb{N}^m$ by $(n_1, \dots, n_m) \preceq (p_1, \dots, p_m)$ if and only if $n_i \leq p_i$ for all $i, 1 \leq i \leq m$.
	
	For $i$, $1 \leq i \leq m$, let $\Gamma^+(P_i) := H(P_i)$, and for $l \geq 2$, let 
	
	\begin{multline}
		\Gamma^+(P_1, \dots, P_m):= \{ v \in \mathbb{Z}^{+l}: v\ is\ minimal\ in\ \{w \in H(P_{i_1}, \dots, P_{i_l}): v_i = w_i \}\\
		for\ some\ i, 1 \leq i \leq l \} 
	\end{multline}
	
	For each $I \subseteq \{1, \dots, m \}$ let $\iota_I$ denote the natural inclusion $ \mathbb{N}^l \to \mathbb{N}^m$ into the coordinate indexed by $I$. The minimal generating set of $H(P_1, \dots, P_m)$ is
	
	$$ \Gamma(P_1, \dots, P_m) := \bigcup_{l=1}^m \bigcup_{\substack{I=\{i_1, \dots, i_l \} \\ i_1< \dots<i_l}} \iota_I(\Gamma^+(P_{i_1}, \dots, P_{i_l}))$$
	
	The Weierstrass semigroup is completely determined by the minimal generating set $\Gamma(P_1, \ldots, P_m)$: if $1 \leq m < \#\F_q $, then
	\begin{equation}
		H(P_1, \ldots, P_m)=\{ lub \{ v_1, \ldots, v_m \}:v_1, \ldots, v_m \in \Gamma(P_1, \ldots, P_m) \}. \label{Wsg}
	\end{equation}
	
	with:
	$$ lub \{ v^{(1)}, \ldots, v^{(t)} \} := (max \{v^{(1)}_1, \ldots, v^{(t)}_1 \}, \ldots, max\{ v^{(1)}_n, \ldots, v^{(t)}_n \}).$$
	The \emph{least upper bound} of $v^{(1)}, \ldots, v^{(t)} \in \mathbb{N}^n$.

	\begin{proposition} \label{WsgToNs}
		Let $A= \sum_{i=1}^m \alpha_i P_i$ be an effective divisor of degree $g$. If $\gamma \nleq \alpha$ for all $\gamma \in \Gamma(P_1, \dots, P_m)$, then $A$ is non-special.
	\end{proposition}
	
	\begin{proof}
		Let $A = \sum_{i=1}^{m} \alpha_i P_i \in {\div}(\FF/\F_q)$ be an effective divisor and $ \sum_{i=1}^{m} \alpha_i=g$. Suppose $A$ is special. Then $dim(A) > deg(A)+1-g=1$. Since $dim(A) \geq 2$, there exists $w \in H(P_1, \ldots, P_m)$ such that $w \leq \alpha$. According to \eqref{Wsg}, there exist $v_1,\ldots,v_m \in \Gamma(P_1, \ldots, P_m)$ satisfying $v_1 \leq lub\{v_1, \ldots, v_m \}= w \leq \alpha$, the contrapositive of the statement completes the proof.
	\end{proof}

	\begin{proposition} \label{WsgKe}
		Let $\FF/\F_q(y)$ be the Kummer extension defined by
		$$ x^m = \prod_{i=1}^{r}(y-\alpha_i)$$
		and let $P_i$ the place associated with $y-\alpha_i$. Then
		$$\Gamma^+(P_1)= \mathbb{N} \setminus \left\{ mk+j | 1 \leq j \leq m-1- \left\lfloor\frac{m}{r}\right\rfloor, 0 \leq k \leq r-2-\left\lfloor\frac{rj}{m}\right\rfloor \right\}$$
		and for $2 \leq l \leq r-\left\lfloor\frac{r}{m}\right\rfloor$ , $\Gamma^+(P_1, \ldots, P_l)$ is given by 
		$$ \left\{ (ms_1+j, \ldots, ms_l+j) | 1 \leq j \leq m-1 \left\lfloor\frac{m}{r}\right\rfloor, s_i \geq 0, \sum_{i=1}^{l}s_i=r-l- \left\lfloor\frac{rj}{m}\right\rfloor \right\} $$
	\end{proposition}
	
	\begin{proof}
		See \cite[Theorem 3.2 ]{camaqu} and \cite[Theorem 10]{yahu}
	\end{proof}

	\begin{lemma} \label{ari}
		Let $r,m \in \mathbb{Z}^+$ be relatively prime.
		\begin{enumerate}
			\item Let $1 \leq j \leq m-1$ and set $t=r\ \mbox{mod}\ m$. then
			$$ \left\lfloor\frac{r(j+1)}{m}\right\rfloor - \left\lfloor\frac{rj}{m}\right\rfloor = \left\{\begin{array}{ll}
				\left\lfloor\frac{r}{m}\right\rfloor+1 & j =\left\lfloor\frac{km}{t}\right\rfloor,\ 1\leq k \leq t-1 \\
				\left\lfloor\frac{r}{m}\right\rfloor & \mbox{otherwise.}
			\end{array}
			\right.  $$
			\item If $t<m$, then
			$$ \sum_{k=1}^{t-1}\left\lfloor\frac{km}{t}\right\rfloor = \frac{(m-1)(t-1)}{2}. $$
		\end{enumerate}
	\end{lemma}
	
	\begin{proof}
	See \cite[Lemma 7]{caloma}.
	\end{proof}

	\begin{theorem}\label{DivDeggKe}
		Let $\FF/\F_q(y)$ by the Kummer extension defined by
		
		$$x^m = \prod_{i=1}^r (y-\alpha_i) $$
		
		where $\alpha_i \in \F_q$ and $ (r,m)=1$. For $1 \leq j \leq m-1- \left\lfloor\frac{m}{r}\right\rfloor$, define the following values:
		\begin{itemize}
			\item $l_j = r-\left\lfloor\frac{rj}{m}\right\rfloor$.
			\item $s_j = l_j -l_{j+1}$ if $j < m-1 - \left\lfloor\frac{m}{r}\right\rfloor$ and $s_{m-1 - \left\lfloor\frac{m}{r}\right\rfloor} = l_{m-1 - \left\lfloor\frac{m}{r}\right\rfloor}-1 = max \{1, \left\lfloor\frac{m}{r}\right\rfloor \}$.
			
		\end{itemize}
		Then $A$ is an effective non-special divisor of degree $g$ with support contained in the set $\{ P_{0b}: \prod_{i=1}^r(b-\alpha_i)=0 \}$ if and only if
		
		$$A = \sum_{j=1}^{m-1 - \left\lfloor\frac{m}{r}\right\rfloor} j \sum_{i=1}^{s_j} P_{0b_{j_i}}.$$
		
		In particular, if $r < m$, 
		
		$$A = \sum_{j=1}^{r-1}  \left\lfloor\frac{jm}{r}\right\rfloor P_{0b_j}. $$
	\end{theorem}
	
	\begin{proof}
		First, note that
		$$ A = \sum_{j=1}^{m-1-\left\lfloor\frac{m}{r}\right\rfloor} j D_j = \sum_{j=1}^{m-1-\left\lfloor\frac{m}{r}\right\rfloor} j \sum_{i=1}^{s_j} P_{0b_{j_i}}$$
		where 
		$$ D_j = \left\{ \begin{array}{ll} \sum_{i=1}^{s_j} P_{0b_{j_i}} & \mbox{if}\ s_j > 0  \\
			0 & \mbox{if}\ s_j = 0 
		\end{array} \right. $$
		for $ 1 \leq j \leq m-1-\left\lfloor\frac{m}{r}\right\rfloor$.
		We will prove that $A$ is of degree $g$. Let $t=r\ \mbox{mod}\ m$, we have
		\begin{align*} deg(A) & = \sum_{i=1}^{m-1-\left\lfloor\frac{m}{r}\right\rfloor} j s_j   \\
			& =  \sum_{i=1}^{m-1} j \left\lfloor\frac{r}{m}\right\rfloor + \sum_{k=1}^{t-1} \left\lfloor\frac{km}{t}\right\rfloor  \\
			& =  \frac{(m-1)m}{2}\left\lfloor\frac{r}{m}\right\rfloor+ \frac{(m-1)(t-1)}{2} \\
			& = \frac{(m-1)}{2} \left( m \left\lfloor\frac{r}{m}\right\rfloor+t-1 \right) \\
			& =  \frac{(m-1)(r-1)}{2} = g   
		\end{align*}  
		where the second equality follows from Lemma \ref{ari} (1) and the third one from Lemma \ref{ari} (2). Therefore, $A$ is effective of degree $g$. Now, we are going to prove that $A$ is non-special using Proposition \ref{WsgToNs}.\\
		Take $v \in \mathbb{N}^{l_1-1}$ such that $A= \sum_{i=1}^{l_1-1} v_i P_i$. Since $v_i \leq m-1-\left\lfloor\frac{m}{r}\right\rfloor$ for all $i$, then by Proposition \ref{WsgKe} $v_i < w$ for any $w \in \Gamma^+(P_i)$, and so
		$$ \iota_{\{i\}}(w) \nleq v $$ 
		for any $w \in \Gamma^+(P_i)$. Take $w \in \Gamma^+(P_{i_j} | j\in I \subset \{1, \ldots, l_1-1 \})$. If for some $i$, $w_i > m-\left\lfloor\frac{m}{r}\right\rfloor$, then $w \nleq v$, so assume $w=(k, \ldots, k)$ for some $1 \leq k \leq m-1-\left\lfloor\frac{m}{r}\right\rfloor$. By Proposition \ref{WsgKe}, we know that $\#I = l_k$ and the number of entries of $v $ greater or equal than $k$ are $\sum_{i=k}^{m-1-\left\lfloor\frac{m}{r}\right\rfloor} s_i=l_k-1$, therefore $\iota_I(w) \nleq v$ for any $I$ of cardinality $l_k$. This concludes that $w \nleq v$ for all $w \in \Gamma(Supp(A))$ and then $A$ is non-special. \\
		Set $\gamma=m-1-\left\lfloor\frac{m}{r}\right\rfloor$ and choose $B$ an effective non-special divisor of degree $g$ supported on $Supp((x))$. If $v_P(B) \geq \gamma + 1$ for some $P$, then $\iota(\gamma +1) \leq B$, contradicting the non-specialty of $B$. \\
		Write $B= \sum_{j=1}^{\gamma }j D_j$ where $D_j$ is zero or is the sum of distinct rational places of degree 1 and $Supp(D_j) \cap Sup(D_h) = \varnothing$ for $j \neq h$. \\
		Observe that $\#Supp(B) \leq l_1-1 <r= \#Supp((x)_0)$. For $D_{\gamma}$ we know
		$$ deg(D_\gamma) \leq l_\gamma -1 = s_\gamma. $$
		Similarly, for $1 \leq h \leq \gamma$, we know
		$$ \sum_{j=h}^{\gamma} deg(D_j) \leq deg(D_h) + \sum_{j=h+1}^{\gamma} deg(D'_j) \leq l_h-1, $$
		so take $D'_h \geq D_h$, $Supp(D_h) \subseteq  Supp((x)_0) \setminus Supp(B+\sum_{j=h+1}^{\gamma} D'_j) $ such that
		$$ \sum_{j=h}^{\gamma} D'_h= \sum_{j=h}^{\gamma} \# Supp(D'_j)= l_h-1. $$
		From this construction, it is clear
		$$ deg(D'_h)=s_h $$
		and so 
		$$ g= deg(B) \leq \sum_{j=1}^{\gamma}j.deg(D'_j)= \sum_{j=1}^{\gamma} j s_j= g. $$
		Then $D'_h=D_h$ for any $1 \leq h \leq \gamma$ and $B$ has the desired form.\\
		In the case $r<m$, By lemma \ref{ari} we have $s_j=1$ if $j= \left\lfloor\frac{km}{r}\right\rfloor$ and $0$ otherwise and then $D_j =P_k $ or $D_j=0$.
	\end{proof}

	\begin{corollary}
		On the norm-trace curve given by $ y^{q^{r-1}} + y^{q^{r-2}} + \dots + y = x^{\frac{q^r-1}{q-1}}$ over $\F_{q^r}$, any effective non-special divisor of degree $g$ supported by points $P_{0b}$ is of the form
		
		$$\sum_{i=1, q \nmid i}^{\frac{q^r-1}{q-1}-2} i P_{0b_i}.$$
	\end{corollary}

	\begin{proof}
		Take $u = \frac{q^r-1}{q-1}$. Given
		$$ \left\lfloor\frac{(r-1)u}{q^{r-1}}\right\rfloor= u-2 $$
		it is enough to prove that $\left\lfloor\frac{jm}{r}\right\rfloor$ cannot be divisible by $q$, since
		$$ u-2- \# \{ i \in \{ 1, \ldots, u-2 \} |\ q|i \} = u-2-\frac{u-1}{q}+1=\frac{u-1}{q}(q-1)=q^{r-1}-1.$$
		and then $A$ should be $\sum_{j=1}^{q^{r-1}-1} \left\lfloor\frac{ju}{q^{r-1}}\right\rfloor P_j$, implying that $A$ is non-special of degree $g$. By Theorem \ref{DivDeggKe}, any other divisor with these characteristics should be of this form.\\
		Thus, we will prove $q \nmid \left\lfloor\frac{ju}{q^{r-1}}\right\rfloor$ for any $1\leq j \leq q^{r-1}-1$. If $j$ is such that $\left\lfloor\frac{ju}{q^{r-1}}\right\rfloor=qk$ for some $1 \leq k \leq \frac{u-1}{q}-1$, then
		$$ ju=q^rk+z=u(q-1)k+k+z. $$
		This expression implies that $u|k+z$, but
		$$ k+z < \frac{u-2}{q}+q^{r-1}=\frac{uq-1}{q}<u. $$
		Then no such $k$ exists, and the conclusion follows. 
	\end{proof}

	\begin{corollary} \label{DivNsgHe}
		On the Hermitian curve $y^q + y - x^{q+1}=0$ over $\FF_{q^2}$, any effective non-special divisor of degree $g$ with support contained in $\{ P_{0b_i}: 1 \leq i \leq q \}$ is of the form
		
		$$ A = \sum_{i=1}^{q-1} i P_{0b_i}.$$
	\end{corollary}
	
	Using \ref{gtog-1} and \ref{DivDeggKe}, we have the following explicit construction.
	
	\begin{theorem} \label{DivDegg-1Ke}
		Let $\FF/\F_q(y)$ by the Kummer extension defined by
		
		$$x^m = \prod_{i=1}^r (y-\alpha_i) $$
		
		where $\alpha_i \in \F_q$ and $ (r,m)=1$. Then 
		
		$$A = \sum_{j=1}^{m-1 - \left\lfloor\frac{m}{r}\right\rfloor} j \sum_{i=1}^{s_j} P_{0b_{j_i}} - P.$$
		
		is a non-special divisor of degree $g-1$ for all $P \in \{P_{ab} | a \neq 0\ or\ b \neq b_{j_i} \} \cup \{P_\infty \}$. In particular there exist non-special divisors of degree $g-1$ supported on $ supp((x)_0) \cup \{ P_{ab} \}$. for any $a \neq 0$.
	\end{theorem}

	\begin{proof}
		Note that $A+P_{ab}$ is non-special of degree $g$ by Theorem \ref{DivDeggKe} and, by Lemma \ref{gtog-1}, we have $A$ is non-special too. Given
		$$ \#Supp(A)= r- \left\lfloor\frac{r}{m}\right\rfloor -1 \leq r-1,$$
		we can take $P \in Supp(x) \setminus Supp(A) = \varnothing.$
	\end{proof}
	
	To complete this section, \cite[Lemma 4.1]{camarqu} provides the explicit form of the non-special divisors of degree $g$ and $g-1$ in the case of a curve $\FF/\F_{q^2}$ defined by 
	
	\begin{equation}
	y^{q+1} = x^2+x   \label{KEs}		
	\end{equation}
	
	\begin{lemma} \label{FEKEs}
		Let $q$ be odd, and $\FF/\F_{q^2}$ be the function field defined by \eqref{KEs}. Let $P \in \{P_{ab} \in \bP(\FF/\F_{q^2}) \mid 2a + 1 \neq 0\}$ be a rational place of $\FF$. Then, $gP$ is a non-special divisor of degree $g$. In particular, $gP - P'$ is a non-special divisor of degree $g-1$ for all rational places $P' \in \bP(\FF/\F_{q^2})$ distinct from $P$.
	\end{lemma}
	
	\begin{proof}
		The hyperelliptic involution of $y^{q+1} = x^2 + x$ is given by $\varphi(x,y) = (-x-1, y)$, so the fixed points of $\varphi$ are those $(a, b) \in \mathbb{F}_{q^2} \times \mathbb{F}_{q^2}$ such that $2a + 1 = 0$. Let $P \in \{P_{ab} \in \bP(\FF/\F_{q^2}) \mid 2a + 1 \neq 0\}$. We nkow that if $P$ is a rational place of an hyperelliptic function field which is not fixed by the hyperelliptic involution, then $H(P)= \{ 0,g+1,g+2,\ldots \}$ (see \cite[Satz 8]{sh}). Therefore $\mathcal{L}(gP) = \{0\}$ and $gP$ is a non-special divisor of degree $g$. Finally, given a rational place $P' \in \bP(\FF/\F_{q^2}) \setminus \{P\}$, we have that $gP - P'$ is also a non-special divisor by \cite[Lemma 3]{balb}.
	\end{proof}

	\subsection{Example on an asymptotically good tower}
	
		A tower $F_1 \subseteq F_2 \subseteq F_3 \subseteq \ldots $ of algebraic function fields over a finite field $\F_q$ is said to be asymptotically good if
	$$ \lim_{m \to \infty} \frac{B_1(F_m/\F_q)}{g(F_m / \F_q)} > 0. $$
	
	We will consider the following tower $\mathcal{F} = (F_m)_{m \geq 1} $ of function fields $\F_m / K$ when $K = \F_{q^2}$ : 
	$$ F_m = K(x_1, \ldots, x_m)\ \mbox{with}\ x_{i+1}^q + x_{i+1} = \frac{x_i^q}{x_i^{q-1}+1}\ \mbox{for}\ i= 1, \ldots ,m-1. $$
	
	This tower was studied in \cite{gast}:
	
	\begin{proposition}
		\begin{enumerate}
			\item[\rm{i})] For all $m \geq 2$, the extension $F_m/F_{m-1}$ is Galois extension of degree $q$.
			
			\item[\rm{ii})] The pole of $x_1$ in $F_1$ is totally ramified in $F_m/T_1$, $\mathit{i.e.}$,
			$$ (x_1)_{\infty}^{F_m} = q^{m-1}.P_\infty^{(m)} $$
			with a place $P_\infty^{(m)} \in \bP_1(F_m)$ of degree one.
			
			\item[\rm{iii})] The genus $g(F_m)$ is 
			$$ g(F_m) = \left\{\begin{array}{ll}
				(q^{m/2}-1)^2 & \mbox{if}\ m \equiv 0\ \mbox{mod}\ 2 \\
				(q^{(m+1)/2}-1)(q^{(m-1)/2}-1) & \mbox{if}\ m \equiv 1\ \mbox{mod}\ 2 \\
			\end{array}
			\right.  $$
		\end{enumerate}
	\end{proposition}
	
	\begin{proof}
		(i), (ii) See \cite[Lemma 3.3]{gast}. (iii) See \cite[Remark 3.8]{gast}
	\end{proof}
	
	\begin{definition}
		For $m \geq 1$, let
		$$ c_m = \left\{\begin{array}{ll}
			q^m-q^{m/2} & \mbox{if}\ m \equiv 0\ \mbox{mod}\ 2 \\
			q^m-q^{(m+1)/2} & \mbox{if}\ m \equiv 1\ \mbox{mod}\ 2 \\
		\end{array}
		\right.  $$
	\end{definition}
	
	\begin{definition}
		For $1 \leq j \leq m$ we define
		$$ \pi_j = \prod_{i=1}^{j} (x_i^{q-1}+1) $$
		and
		$$ \mathtt{L}_j^{(m)} = \{ P \in \bP(F_m)\ |\ P\ \mbox{is a zero of}\ x_i^{q-1}+1,\ \mbox{for some}\ i \in \{ 1, \ldots, j \} \}. $$
	\end{definition}
	
	\begin{lemma} \label{PrincPi}
		\begin{enumerate}
			\item[\rm{i})] Let $1 \leq j \leq m$. Then the principle divisor of $\pi_j$ is given by 
			$$ (\pi_j)^{F_m} = C_j^{(m)}-(q^m-q^{m-j})P_\infty^{(m)}, $$
			where $C_j^{(m)} \geq 0$ is a divisor of $F_m$ with
			$$ supp(C_j^{(m)}) = \mathtt{L}_j^{(m)}. $$
			\item[\rm{ii})] Let $1 \leq j \leq m-1$ and $0 \leq e \leq q-1$. Then the principal divisor of $\pi_jx_{j+1}^e$ in $F_m$ is given by 
			$$ (\pi_jx_{j+1}^e)^{F_m} = D_{j,e}^{(m)}- (q^m-q^{m-j}+eq^{m-j-1})P_\infty^{(m)},$$
			where $D_{j,e}^{(m)} \geq 0$ is a divisor of $F_m$ with
			$$ \mathtt{L}_j^{(m)} \subseteq supp(D_{j,e}^{(m)}).$$
		\end{enumerate}
	\end{lemma}
	
	\begin{proof}
		See \cite[Lemma 3.4]{pestto}
	\end{proof}
	
	\begin{definition}
		For $1 \leq j \leq m$, let
		$$ A_j^{(m)} = \sum_{P \in \mathtt{L}_j^{(m)}} P. $$
	\end{definition}
	
	\begin{remark}
		From Lemma \ref{PrincPi} we have: $\pi_j \in \mathcal{L}((q^m-q^{m-j})P_\infty^{(m)}-A_j^{(m)})$.
	\end{remark}
	
	\begin{proposition} \label{OneDim}
		For $1 \leq j \leq m$,
		$$ \mathcal{L}((q^m-q^{m-j})P_\infty^{(m)}-A_j^{(m)}) = <\pi_j>; $$
		$\mathit{i.e.}$, the space $ \mathcal{L}((q^m-q^{m-j})P_\infty^{(m)}-A_j^{(m)})$ is one-dimensional.
	\end{proposition}
	
	\begin{proof}
		See \cite[Proposition 3.6]{pestto}
	\end{proof}
	
	\begin{lemma} \label{DegAm} \cite[Lemma 3.7]{pestto}
		Let $1 \leq j \leq m/2$. Then
		$$ deg(A_j^{(m)}) = q^j-1 $$
	\end{lemma}
	
	\begin{proof}
		Let $\mathcal{A}_i^{(m)} = \{ P \in \bP(F_m)\ |\ P\ \mbox{is a zero of}\ x_i^{q-1}+1  \}$. It follows from \cite[Lemma 3.6]{gast} that for $1 \leq i \leq m/2$,
		$$ deg\left(\sum_{P \in \mathcal{A}_i^{(m)}}P\right)  = (q-1)q^{i-1}. $$ 
		Since 
		$$ A_j^{(m)} = \sum_{i=1}^{j} \sum_{P \in \mathcal{A}_i^{(m)}} P, $$
		we obtain 
		$$ deg(A_j^{(m)}) = \sum_{i=1}^{j} (q-1)q^{i-1} = q^j-1. $$
	\end{proof}
	
	\begin{definition}
		We define a divisor $A^{(m)}$ of $F_m$ as follows: $A^{(1)}=0$ and, for $m \geq 2$,
		$$ A^{(m)} = A_j^{(m)}\ \ \ \mbox{with}\ j = \left\{\begin{array}{ll}
			\frac{m}{2} & \mbox{if}\ m \equiv 0\ \mbox{mod}\ 2 \\
			\frac{m-1}{2} & \mbox{if}\ m \equiv 1\ \mbox{mod}\ 2 \\
		\end{array}
		\right.  $$
	\end{definition}
	
	\begin{lemma} \label{ExiAGT} \cite[Lemma 3.9]{pestto}
		\begin{enumerate}
			\item[\rm{i})] $deg(A^{(m)}) = c_m - g(F_m) $.
			\item[\rm{ii})] $dim(c_m P_\infty^{(m)}- A^{(m)})=1$.
		\end{enumerate}
		$\mathit{i.e.}$, $c_m P_\infty^{(m)}- A^{(m)}$ is non-special of degree $g(F_m)$.
	\end{lemma}
	
	\begin{proof}
		For $m=1$, all assertions are obvious since $c_1 = g(F_1) = 0$ and $A^{(1)}=0$. Let $m \geq 2$. 
		\begin{itemize}
			\item If $m \equiv 0\ \mbox{mod}\ 2$. Then
			$$ c_m = q^m-q^{m/2}\ \ \ \mbox{and}\ \ \ g_m = (q^{m/2}-1)^2.$$
			Hence $$ c_m - g_m = q^{m/2}-1= deg(A^{(m)}), $$
			By Lemma \ref{DegAm}. On the other hand, we have
			$$ \mathcal{L}(c_m P_\infty^{(m)}-A^{(m)}) = \mathcal{L}((q^m-q^{m/2})P_\infty^{(m)}-A_{m/2}^{(m)}) = <\pi_{m/2}>, $$
			By Proposition \ref{OneDim}.
			\item If $m \equiv 1\ \mbox{mod}\ 2$. The proof is similar.
		\end{itemize}
	\end{proof}
	
	\subsection{Example of construction with a sufficiently large number of points}
	
	This section presents the tools proposed by H. Randriam in \cite{ra} which allow for the construction of non-special divisors of degree $g-1$.
	
		\begin{lemma} \label{Plu1}
		Let $\CC$ be a curve of genus $g$ over $\F_q$, and $\mathcal{S} \subset \CC(\F_q)$. 
		\begin{enumerate}
			\item Let $A$ be an $\F_q$-rational divisor on $\CC$ such that 
			$$ i(A)=dim(A)-(deg(A)+1-g) \geq 1. $$
			Suppose that for all $P \in \mathcal{S}$ we have $dim(A+P) > dim(A)$. Then
			\begin{equation}
				\#\mathcal{S} \leq g-dim(A).
			\end{equation}
			(if $deg(A)=-1$, then, we have also $\#\mathcal{S} \leq 1$).
			\item Let $B$ a $\F_q$-rational divisor on $\CC$ such that $dim(B) \geq 1$. Suppose that for all $P \in \mathcal{S}$ we have $dim(B-P) > dim(B)-1$. Then
			\begin{equation}
				\#\mathcal{S} \leq deg(B)+1-dim(B).
			\end{equation}
			(if $deg(B)=2g-1$, we have also $\#\mathcal{S}\leq 1$).
		\end{enumerate}
	\end{lemma}
	
	\begin{proof}
		See \cite[Lemma 1]{ra}.
	\end{proof}
	
	The main result of this section is based on the following Lemma. \\
	For all $q>1$ and all integer $n \geq 2$ we define:
	\begin{align*}
		G_q(n) &= \sum_{k=1}^{n-2} \frac{(q^{n-k}-1)(q^{n-k-1}-1)}{(q^n-1)(q^{n-1}-1)} \\
		&=\frac{1}{q^2-1}-\frac{1- \frac{(q-1)n}{q^n}-1}{(q-1)(q^{n-1}-1)}
	\end{align*}
	
	\begin{lemma} \label{Plu2}
		Let $\CC$ be a curve of genus $g$ over $\F_q$, and $\mathcal{S} \subset \CC(\F_q)$. 
		\begin{enumerate}
			\item Let $A$ be an $\F_q$-rational divisor on $\CC$ such that $deg(A) \geq -2$ and
			$$ i(A)=dim(A)-(deg(A)+1-g) \geq 2. $$
			Suppose that for all $P \in \mathcal{S}$ we have $dim(A+2P) > dim(A)$. Then
			\begin{equation}
				\#\mathcal{S} \leq 3g+3+deg(A)-3dim(A).
			\end{equation}
			and
			\begin{equation}
				\#\mathcal{S} \leq \left( 1+ \frac{q^{i(A)-2}-1}{q^{i(A)}-1} \right)^{-1} (6g-6-2deg(A)-2G_q(i(A)) \cdot \#\CC(\F_q).
			\end{equation}
			more generally, for all integers $w$ such that $2 \leq w \leq i(A)$,
			\begin{multline}
				\#\mathcal{S} \leq (i(A)-w)+\left( 1+ \frac{q^{w-2}-1}{q^{w}-1} \right)^{-1} (6g-6-2deg(A)-4(i(A)-w)\\ -2G_q(w) \cdot \#\CC(\F_q).
			\end{multline}
			\item Let $B$ an $\F_q$-rational divisor on $\CC$ such as $deg(B) \leq 2g$ and $dim(B) \geq 2$. Suppose that for all $P \in \mathcal{S}$ we have $dim(B-2P) > dim(B)-2$. Then
			\begin{equation}
				\#\mathcal{S} \leq 2deg(B)+2g+4-3dim(B).
			\end{equation}
			and
			\begin{equation}
				\#\mathcal{S} \leq \left( 1+ \frac{q^{dim(B)-2}-1}{q^{dim(B)}-1} \right)^{-1} (2deg(B)+2g-2-2G_q(dim(B)) \cdot \#\CC(\F_q).
			\end{equation}
			more generally, for all integers $w$ such that $2 \leq w \leq dim(B)$,
			\begin{multline}
				\#\mathcal{S} \leq (l(B)-w)+\left( 1+ \frac{q^{w-2}-1}{q^{w}-1} \right)^{-1} (2deg(B)+2g-2-4(dim(B)-w)\\ -2G_q(w) \cdot \#\CC(\F_q).
			\end{multline}
		\end{enumerate}
	\end{lemma}
	
	\begin{proof}
		See \cite[Lemma 2]{ra}.
	\end{proof}
	
	For the following results, we introduce two functions $f_{1,\CC}$ and $f_{2,\CC}$  defined over $\Z$ by:
	
	$$ f_{1,\CC}(a) = \left\{\begin{array}{lll}
		1 & \mbox{if}\ a=-1 \\
		g & \mbox{if}\ 0\leq a \leq g-2 \\
		0 & \mbox{otherwise}
	\end{array}
	\right.  $$
	and 
	$$
	f_{2,\CC}(a) = \left\{\begin{array}{llll}
		g & \mbox{if}\ a=g-2 \\
		\underset{2\leq w \leq g-1-a}{min} \lfloor (g-1-a-w) + \left(1+\frac{q^{w-2}-1}{q^w-1} \right)^{-1} & \mbox{if}\ -2\leq a \leq g-3 \\
		\ \ \ \ \ \ \ \ \  (2g-2 +2a+4w-2G_q(w) \#\CC(\F_q))\rfloor & \\
		0 & \mbox{otherwise}
	\end{array}   
	\right. 
	$$
	
	\begin{definition}
		Let $\CC(\F_q)$ be a curve of genus $g \geq 1$. A divisor $D$ on $\CC$ is called \emph{ordinary} if
		$$ dim(D)=max(0,deg(D)+1-g). $$
		otherwise $D$ is called \emph{exceptional}.
	\end{definition}
	
	\begin{lemma} \label{RanL5}
		Let $A$ be a divisor on $\CC$, $\mathcal{S} \subset \CC(\F_q) $ a set of points and $s \in \{ 1, 2\}$. Assume that $A$ is ordinary, and $A+sP$ is exceptional for all $ P \in \mathcal{S}$. Then
		\begin{equation}
			\#S \leq f_{s,\CC}(deg(A)).
		\end{equation}
	\end{lemma}
	
	\begin{proof}
		Let $a=deg(A)$. 
		\begin{enumerate}
			\item[\rm{i})] Assume that $s=1$.
			\begin{itemize}
				\item If $a \leq -2$ or $a \geq 2g-2$, then $A+P$ is ordinary, thus $\mathcal{S}$ is empty, and we have $f_{1,\CC}(a)=0$.
				\item If $g-1 \leq a \leq 2g-3$, then  ordinary $A$ mean that $dim(A)=a+1-g$, thus $A+P$ is ordinary by Riemann-Roch, and we have the same conclusion. 
				\item If $-1 \leq a \leq g-2$, then  ordinary $A$ mean that $dim(A)=0$, and exceptional $A+P$ mean that $dim(A+P)\geq 1$. We conclude by Lemma \ref{Plu1}-1.
			\end{itemize}
			
			\item[\rm{ii})] Assume that $s=2$.
			\begin{itemize}
				\item If $a \leq -3$ or $a \geq 2g-1$, then $A+2P$ is ordinary, thus $\mathcal{S}$ is empty.
				\item If $g-1 \leq a \leq 2g-2$, then ordinary $A$  mean that $dim(A)+a+1-g$, thus $A+2P$ is ordinary by Riemann-Roch, and we have the same conclusion.
				\item If $a=g-2$, then  ordinary $A$ mean that $dim(A)=0$, and exceptional $A+2P$ mean that $dim(A+2P)\geq 1$. By Lemma \ref{Plu1}-1, we have $\#S\leq g=f_{2,\CC}(a)$.
				\item If $-2 \leq a \leq g-3$, then ordinary $A$ mean that $dim(A)=0$, and exceptional $A+2P$ mean that $dim(A+2P) \geq 1$. We conclude by Lemma \ref{Plu2}-1.
			\end{itemize}
		\end{enumerate}
	\end{proof}

	\begin{proposition} \label{RanP6}
		Let $\CC(\F_q)$ be a curve of genus $g$, $r\geq1$ an integer, $s_1, \ldots, s_r \in \{ 1, 2\}$,  $\F_q$-rational divisors $T_1, \ldots, T_r $ on $\CC $, and $d \in \Z$ an integer. The degree of $T_i$ is denoted $t_i=deg(T_i)$. Let $D_0$ be a $\F_q$-rational divisor on $\CC$ of degree $deg(D_0)=d_0 \leq d$, such that $s_i D_0 -T_i$ is ordinary. Finelly, let $\mathcal{S} \subset \CC(\F_q)$ be a set of points such that
		\begin{equation} 
			\#\mathcal{S} > \underset{d_0\leq d' < d}{max} \sum_{i=1}^{r} f_{s_i,\CC}(s_i d' - t_i). \label{InP6}
		\end{equation}
		Then, there exists an $\F_q$-rational divisor $D$ on $\CC$ of degree $deg(D)=d$, such that $s_i D -T_i$ is ordinary. Furthermore, $D$ can be chosen such that $D-D_0$ is effective and supported by points in $\mathcal{S}$.
	\end{proposition}
	
	\begin{proof}
		We will built $D$ incrementally. Let $d'$ such that $d_0 \leq d' < d$. Assume that a divisor $D'$ of degree $d'$ is already built, with $s_i D' -T_i$ ordinary and $D'-D_0$ effective supported in $\mathcal{S}$. Lemma \ref{RanL5} applied with $s=s_i$ and $A=s_i D'- T_i$ shows that there are no more than $f_{s_i,\CC}(s_i d'-t_i)$ points such that $s_i(D'+P)-T_i$ is exceptional. Applying \eqref{InP6}, we can find $P \in \mathcal{S}$ such that $s_i(D'+P)-T_i$ is ordinary and $(D'+P)-D_0$ effective (by construction) supported in $\mathcal{S}$. We conclude by induction on $d'$.
	\end{proof}
	
	\begin{lemma} \label{RanL8}
		With the previous notations, $f_{s,\CC}(a)$ is an increasing function when $a \leq g-1-s$, their maximum value on $\Z$ is $f_{s,\CC}(g-1-s)=s^2g$.
	\end{lemma}
	
	\begin{proof}
		See \cite[Lemma 8]{ra}.
	\end{proof}
	
	\begin{proposition} \label{RanP19}
		Let $\CC(\F_q)$ be a curve of genus $g$, $r\geq1$ an integer, $s_1, \ldots, s_r \in \{ 1, 2\}$,  $\F_q$-rational divisors $T_1, \ldots, T_r $ on $\CC $. Assume that 
		$$ \#\CC(\F_q) > \sum_{i=1}^{r}(s_i)^2 g.$$
		Then, for all integer $d$, there exists an $\F_q$-rational divisor $D$ on $\CC$ of degree $deg(D)=d$ and supported in $\CC(\F_q)$, such that $s_i D- T_i$ is ordinary.   
	\end{proposition}
	
	\begin{proof}
		We apply Proposition \ref{RanP6} with $\mathcal{S}=\CC(\F_q)$, and Lemma \ref{RanL8}.
	\end{proof}
	
	\begin{corollary} \label{ExDeCons1}
		Let $\CC(\F_q)$ be a curve of genus $g$, $Q$ and $G$ two $\F_q$-rational divisors. We denote by  $k=deg(Q)$ and $n=deg(G)$ their degrees, assume that
		$$ \#\CC(\F_q) > 5g $$
		and
		$$ n \geq 2k +g -1 .$$
		Then, there exists an $\F_q$-rational divisor $D$ on $\CC$ supported in $\CC(\F_q)$, such that $D-Q$ is non-special of degree $g-1$ and $dim(2D-G)=0$. \\
		In particular, if $n=2k+g-1$, then $D-Q$ and $2D-G$ are non-special of degree $g-1$.
	\end{corollary}
	
	\begin{proof}
		We apply Proposition \ref{RanP19} with $r=2$, $s_1=1$, $T_1=Q$, $s_2=2$, $T_2=G$ and $d=k+g-1$.
	\end{proof}
	
	\begin{remark} \label{ExDeCons2}
	Notice that the divisor $D$ in the previous corollary can be built as seen in the proof of Proposition \ref{RanP6}. Below is the summary of the steps with the conditions of Corollary \ref{ExDeCons1}:
	\begin{enumerate}
		\item Let $Q$ and $G$ be two $\F_q$-rational divisors with $deg(Q)=k$ and $deg(G)=2k+g-1$. 
		\item Let $D_0$ be a divisor such that $deg(D_0)=d_0 \leq k+g-1$ and $D_0-Q$, $2D_0-G$ are ordinary.
		\item Build a divisor $D'$ of degree $d'$, $d_0 \leq d' < k+g-1$ such that $D'-Q$, $2D'-G$ are ordinary and $D'-D_0$ effective.
		\item We can find $P \in \CC(\F_q)$ such that $D'+P-Q$, $2(D'+P)-G$ are ordinary.
		\item We reapply step 4 until we obtain the the desired divisor $D$ with the desired degree $k+g-1$. 
	\end{enumerate}
	\end{remark}

\bibliographystyle{plain}
\bibliography{BallKoutPiel_divNS_biblio-tris}

\end{document}